\documentclass[12pt]{amsart}
\usepackage{amsthm, amstext, amsmath, amssymb, graphicx, verbatim, fullpage}

\usepackage{mathtools}
\usepackage{booktabs}
\usepackage{hyperref}
\usepackage{ mathrsfs }
\usepackage{algorithm}
\usepackage{xcolor}
\usepackage{versions}
\usepackage{enumerate}
\usepackage{graphicx}
\usepackage{tabularx}
\usepackage{relsize}

\usepackage{tikz-cd}
\usepackage{colonequals}

\usepackage[utf8]{inputenc}
\usepackage[english]{babel}

\numberwithin{equation}{section}

\newtheorem{thm}{Theorem}[section]
 
\newtheorem{lemma}[thm]{Lemma}
\newtheorem{cor}[thm]{Corollary}
\newtheorem{prop}[thm]{Proposition}
\theoremstyle{definition}
\newtheorem{defn}[thm]{Definition}
\newtheorem{question}[thm]{Question}
\newtheorem{conj}[thm]{Conjecture}
\newtheorem{example}[thm]{Example}
\newtheorem{obs}[thm]{Observation}

\newtheorem{remark}[thm]{Remark}

\newcommand{\A}{\mathbb{A}}
\newcommand{\Z}{\mathbb Z}
\newcommand{\N}{\mathbb N}
\newcommand{\R}{\mathbb R}
\newcommand{\C}{\mathbb C}
\newcommand{\Q}{\mathbb Q}

\newcommand{\shL}{\mathcal{L}}
\newcommand{\PP}{\mathbb{P}}

\newcommand{\ovl}{\overline}

\DeclareMathOperator{\Id}{Id}

\DeclareMathOperator{\PGL}{PGL}

\newcommand{\abs}[1]{\left|#1\right|}

\newcommand{\mc}{\mathcal}

\DeclareMathOperator{\rad}{rad}

\DeclareMathOperator{\Div}{Div}

\DeclareMathOperator{\Top}{top}
\DeclareMathOperator{\Alg}{alg}
\DeclareMathOperator{\arith}{arith}

\DeclareMathOperator{\Ind}{Ind}

\DeclareMathOperator{\Num}{Num}

\newcommand{\RS}{\mathrm{RS}}

\newcommand{\mult}{\mathrm{mult}}

\newcommand{\GLP}{\mathrm{GLP}}

\newcommand{\bk}{\mathbf{k}}
\newcommand{\rmd}{d}

\DeclareMathOperator{\Dom}{Dom}

\begin{document}
 
\title{Degree Growth of Skew Pentagram Maps}
\author{Max Weinreich}

\begin{abstract}
    Skew pentagram maps act on polygons by intersecting diagonals of different lengths. They were introduced by Khesin-Soloviev in 2015 as conjecturally non-integrable generalizations of the pentagram map, a well-known integrable system. In this paper, we show that certain skew pentagram maps have exponential degree growth and no preserved fibration. To formalize this, we introduce a general notion of first dynamical degree for lattice maps, or shift-invariant self-maps of $(\PP^N)^\Z$. We show that the dynamical degree of any equal-length pentagram map is $1$, but that there are infinitely many skew pentagram maps with dynamical degree $4$.
\end{abstract}
\maketitle

\section{Introduction} \label{sect_intro}

\begin{defn}[\cite{schwartz}] \label{def_penta_intro}
    Let $\PP^2$ denote the projective plane over a field $\bk$. A \emph{closed $n$-gon} $v$ is an ordered $n$-tuple of points $(v_1,\ldots, v_n)\in(\PP^2)^n$, cyclically indexed. The \emph{space of closed $n$-gons} is the variety $(\PP^2)^n$. Each pair of distinct points $v_1, v_2 \in \PP^2$ determines a unique line, denoted $\overline{v_1 v_2}$. For each $n\geq 4$, the \emph{pentagram map} on closed $n$-gons is the rational map 
    \[T^{(n)}: (\PP^2)^n \dashrightarrow (\PP^2)^n\]
    that sends a generic $n$-gon $(v_1,\ldots, v_n)$ to $(v_1',\ldots, v_n')$, where for each $i$, 
    \begin{equation} \label{eq_v_classic_intro}
    v_i' = \ovl{v_{i}v_{i+2}}\cap\ovl{v_{i+1}v_{i+3}}.
    \end{equation}
\end{defn}

The pentagram map is an integrable system in many different senses of the term. It admits an invariant fibration \cite{schwartz}, is discrete Liouville-Arnold integrable over $\R$ \cite{OST,OSTclosed}, and is algebraically completely integrable over more general fields \cite{soloviev, weinreich2021}. Further, the iterates of the pentagram map have explicit algebraic formulas of polynomial degree growth \cite{MR4646096,MR4865926}. 

Khesin-Soloviev introduced skew pentagram maps, a generalization of \eqref{eq_v_classic_intro} in which the diagonals are allowed to have different combinatorial lengths \cite{MR3282373}.
To capture the general-case behavior of skew pentagram maps, we work with infinite polygons.

\begin{defn}[\cite{MR3282373}] \label{def_sk_penta_intro}
    An (infinite) \emph{polygon} is a sequence $v \in (\PP^2)^\Z$. The \emph{space of polygons} is the set $(\PP^2)^\Z$. Given $a,b,c,d\in \Z$ distinct, the \emph{skew pentagram map} with parameters $(a,b,c,d)$ on infinite polygons is the partially defined self-map 
    \[T_{a,b,c,d}: (\PP^2)^\Z \dashrightarrow (\PP^2)^\Z,\]
    that sends $v=(v_i)$ to $v'=(v_i')$, where for each $i \in \Z$,
    \begin{equation}\label{eq_sk_intro}
        v'_i=\ovl{v_{i+a}v_{i+b}}\cap \ovl{v_{i+c}v_{i+d}},
    \end{equation}
    whenever all these intersection points are defined.
    When $\abs{b-a} = \abs{d-c}$, we say that $T_{a,b,c,d}$ is \emph{equal-length}. Otherwise, we say that it is \emph{truly skew}. See Figure \ref{fig_classical_vs_skew} and Figure \ref{fig_assortedskew}.

    We also work with skew pentagram maps on closed $n$-gons when $a, b, c, d$ are distinct modulo $n$. We denote these maps
    \[T^{(n)}_{a,b,c,d}: (\PP^2)^n \dashrightarrow (\PP^2)^n.\]
\end{defn}

\begin{figure}[b]
    \centering
    \includegraphics[width=0.75\linewidth]{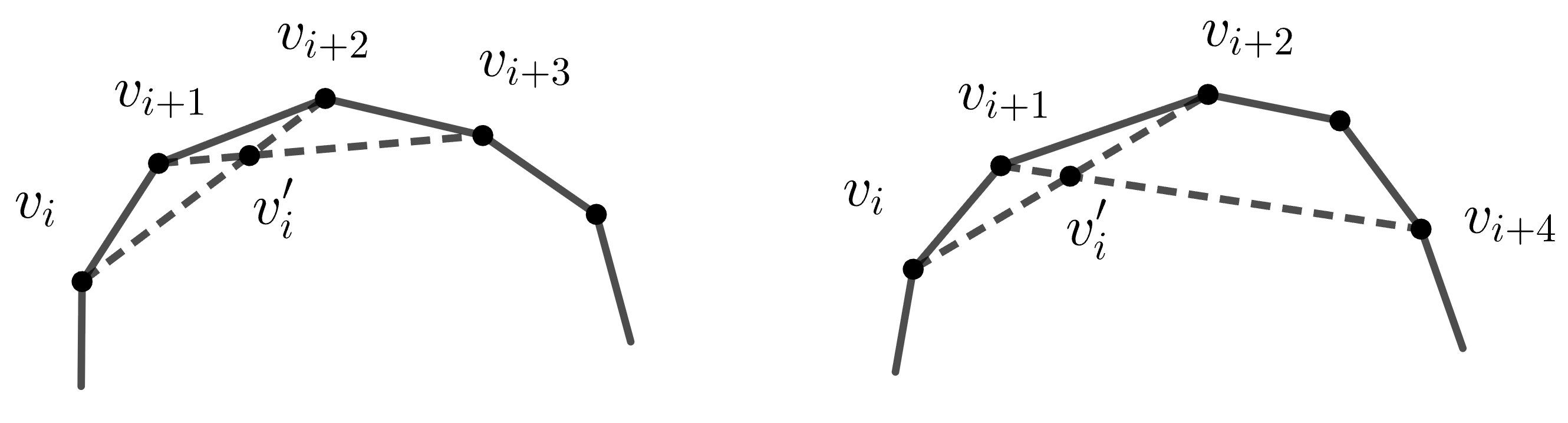}
    \caption{Left: the classical pentagram map $T_{0,2,1,3}$. Right: the skew pentagram map $T_{0,2,1,4}$.}
    \label{fig_classical_vs_skew}
\end{figure}

\begin{figure}[b]
    \centering
    \includegraphics[width=0.75\linewidth]{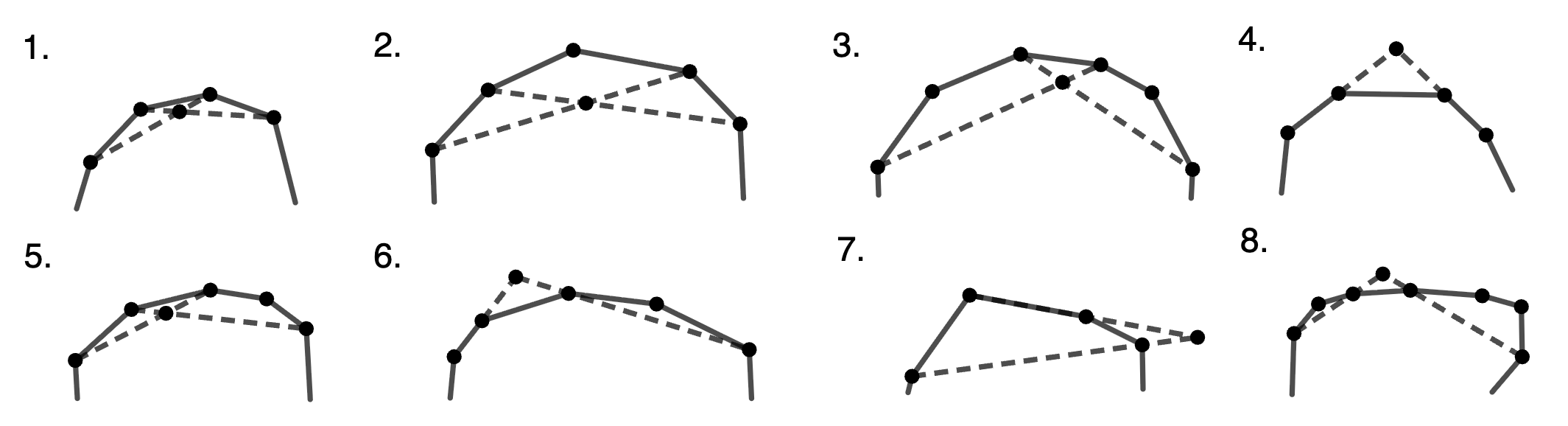}
    \caption{An assortment of skew pentagram maps. In each example, the new $i$-th vertex is the intersection of the dotted lines. The first four examples are equal-length, and the last four are truly skew.}
    \label{fig_assortedskew}
\end{figure}

Equal-length maps include the classical pentagram map $T_{0,2,1,3}$, its inverse $T_{-3,-2,-1,0}$, and the deep-diagonal maps $T_{0,b,1,1+b}$ \cite{textbookcasepentagramrigidity,MR4761767,MR3534837,zou2025spiralstictactoepartitiondeep}.

Khesin-Soloviev introduced skew pentagram maps with an eye towards classifying integrable generalizations of the pentagram map. They computed the first few iterates of randomly chosen polygons with coordinates in $\Q$. These experiments indicated that the number of digits in the denominators of the vertices grew polynomially in the case of equal-length maps, but exponentially for truly skew maps \cite{MR3282373}. Many works since have established various forms of integrability for particular equal-length cases and their higher-dimensional generalizations \cite{MR3356734,izo_khes,khesinsoloviev,MR3438382,izosimov,MR3534837,MR3483131}. However, non-integrability of truly skew maps has yet to be proved. 

In this paper, we continue the study of integrability versus non-integrability of skew pentagram maps from an algebraic point of view. While equal-length maps are invertible (Proposition \ref{prop_inv_Z}), it turns out that truly skew maps may have arbitrarily large generic topological degree (Corollary \ref{cor_unbounded_top}). We therefore consider integrability phenomena that make sense for non-invertible maps: degree growth and existence of invariant fibrations. 

Degree growth is often used as an integrability test in the mathematical physics literature; see e.g. \cite{MR4906745, MR2392894, MR2481234,MR4827588,MR4033822, MR2358970}. In the context of rational self-maps $f$ of varieties over algebraically closed fields, degree growth is described by the \emph{dynamical degree} $\lambda_1(f)$, the exponential degree growth rate \cite{MR1704282,MR4007163, dangsingular}. 

To make sense of degree growth in the general setting of pentagram-like maps, we define \emph{lattice maps}, shift-equivariant self-maps of $(\PP^N)^\Z$ associated to local rules such as \eqref{eq_sk_intro}; see Section \ref{sect_latt} for the full definition. Lattice maps have a different flavor than rational self-maps of varieties due to the way that information diffuses through the domain over time. With lattice maps, the number of variables in the formulas increases under iteration; for instance, each vertex of $T_{a,b,c,d} \circ T_{a,b,c,d}$ depends on $10$ initial vertices in the truly skew case, and $9$ in the equal-length case.
Other examples of lattice maps include the the projective heat map \cite{MR3642552}, the projective evolute map \cite{MR4790973}, and the octahedron recurrence \cite{MR2317336}.

A lattice map $f$ has an \emph{algebraic degree} $\deg f$, defined as the degree of the homogeneous polynomial formula describing the local rule of $f$ when this formula is written without common factors. The \emph{dynamical degree} $\lambda_1(f)$ of $f$ is the exponential degree growth rate  
    \begin{equation}
        \lambda_1(f) = \lim_{m\to \infty} (\deg f^m)^{1/m}.
        \label{eq_dd_intro}
    \end{equation}
One typically expects that $\deg (f^m) = (\deg f)^m$, so $\lambda_1(f) = \deg f$. However, since the iterates must be written in lowest terms, unexpected common factors may lead to lower-than-anticipated degree growth. To calculate dynamical degrees, one must control the appearance of these common factors.

Our main result shows that, while equal-length maps are integrable from the point of view of degree growth, there are infinitely many non-integrable truly skew pentagram maps.

\begin{thm}\label{thm_main}
    Let $a,b,c,d\in \Z$ be distinct integers, and let $T_{a,b,c,d}:(\PP^2_\C)^\Z\dashrightarrow(\PP^2_\C)^\Z$ be the skew pentagram map with parameters $(a,b,c,d)$ over $\C$.
    \begin{enumerate}
        \item If $T_{a,b,c,d}$ is equal-length $(\abs{b-a}= \abs{d-c})$, then $\lambda_1(T_{a,b,c,d}) = 1$. 
        \item There are infinitely many truly skew pentagram maps $T_{a,b,c,d}$ with $\lambda_1(T_{a,b,c,d})=4$. \label{it_thm_main_nonint}
    \end{enumerate}
\end{thm}
Since each vertex of $T_{a,b,c,d}(v)$ depends linearly on $4$ vertices of $v$, there is a quick upper bound $\lambda_1(T_{a,b,c,d}) \leq \deg T_{a,b,c,d} = 4$; see Proposition \ref{prop_upper}. So Theorem \ref{thm_main} shows that infinitely many truly skew maps have maximal algebraic entropy given the natural constraints on the class of skew pentagram maps.

We now reinterpret the experiment \cite{MR3282373} that prompted the present work. If $x\in \PP^N(\Q)$, then we may write $x = [x_0:\cdots:x_N]$, where all $x_j \in\Z$ and $\gcd(x_0,\ldots, x_n) = 1$. The \emph{height} of $x$ is defined as $h(x) = \log \max_j \abs{x_j}$. The \emph{height} of a closed $n$-gon $v\in (\PP^2)^n(\Q)$ with rational coordinates is defined as $\max_i h(v_i)$. This definition extends to algebraic polygons $x \in (\PP^2)^n(\bar{\Q})$; see e.g. \cite{Silverman10}. The \emph{arithmetic degree} of a skew pentagram map $T= T_{a,b,c,d}^{(n)}$ is the exponential height growth rate
\[\alpha_1(T) = \sup_{v\in(\PP^2)^n(\bar{\Q})} \lim_{m\to \infty} h(T^m(v))^{1/m}.\]
The \emph{arithmetic entropy} of $T$ is defined as $h_{\arith} (T) = \log \alpha_1(f)$.

By Matsuzawa-Xie's work on the Kawaguchi-Silverman conjecture, arithmetic and dynamical degrees of maps defined over $\bar{\Q}$ are equal \cite[Theorem 1.2]{MR4958582}. We use this to confirm the predictions of Khesin-Soloviev in an infinite family of examples. 

\begin{cor} \label{cor_height}
    All dominant equal-length maps on closed polygons have arithmetic degree $1$. In contrast, there are infinitely many choices of parameters $(a,b,c,d)$ such that, for infinitely many values of $n$, the skew pentagram map $ T^{(n)}_{a,b,c,d}$ on closed $n$-gons has arithmetic degree $\alpha_1( T^{(n)}_{a,b,c,d}) = 4$. 
\end{cor}

Similarly, applying general entropy comparison theorems of Dinh-Sibony, Guedj, and Dinh-Nguyen-Truong, we have positive topological entropy in these examples (Corollary \ref{cor_rs_consequences}) \cite{dinhnguyen,guedj,MR3436236}.

Based on Khesin-Soloviev's height growth experiments and Theorem \ref{thm_main}, we state the following conjecture.

\begin{conj}
All truly skew pentagram maps $T_{a,b,c,d}$ have dynamical degree $4$.
\end{conj}

The proof of Theorem \ref{thm_main} uses complementary techniques to study the equal-length and truly skew cases. Our analysis of equal-length maps closely follows Affolter-de Tili\`ere-Melotti \cite{MR4646096,MR4865926}. We use Menelaus' Theorem in elementary projective geometry to embed equal-length dynamics in the Schwarzian octahedron recurrence, an integrable system with polynomial degree growth.

To study truly skew maps, we restrict to closed $n$-gons and use techniques from the theory of dynamical degrees of rational maps. However, computing dynamical degrees in higher dimensions is extremely difficult, and the dimension of this domain is $2n$. We need a trick to reduce the dimension of the computation. Since $\lambda_1(T_{a,b,c,d}) \leq 4$ by Proposition \ref{prop_upper}, it suffices to find an invariant subsystem where the map's restriction has dynamical degree at least $4$; we think of this as the ``entropy sandwich'' strategy. We first consider the truly skew map $T_{0,2,1,4}$, which we study via the induced map $\bar{T}_{\RS}$ on the moduli space of rotationally symmetric closed octagons, a surface. Via an appropriate blowup, we construct an algebraically stable model for $\bar{T}_{\RS}$, and find that $\lambda_1(\bar{T}_{\RS}) = 4$, providing the desired lower bound for the lattice map $T_{0,2,1,4}$ on infinite polygons. But many other truly skew pentagram maps happen to agree with $T_{0,2,1,4}$ on closed $8$-gons just by reducing the parameters modulo $8$, so this technique calculates the dynamical degree for a positive-density subset of the space of all possible parameters $(a,b,c,d)$.

Dynamical degrees can sometimes be used to rule out existence of preserved fibrations. Using this, we obtain a different kind of non-integrability for this particular skew pentagram map.

\begin{thm} \label{thm_no_fib}
    The map $\ovl{T}_{\RS}$ induced by the skew pentagram map $T_{0,2,1,4}$ on the moduli space of closed $8$-gons with $4$-fold rotational symmetry does not preserve any nontrivial rational fibration.
\end{thm}

Our analysis of $\bar{T}_{\RS}$ reveals striking similarities to Kaschner-Roeder's study of the heat map on the moduli space of pentagons, an inspiration for this work \cite{KR}. In our framework, their results show that the dynamical degree of the heat map on infinite polygons is $4$ (Corollary \ref{cor_heat}). The reader new to the topic of dynamical degrees will benefit from reading their treatment of the topic. 

In this article, we only study lattice maps on $(\PP^N)^\Z$, but the analogous definition makes sense for $V^G$, where $V$ is a variety and $G$ is a group. It would be interesting to see a definition of dynamical degree for such maps, as well as higher dynamical degrees. As in the case of rational self-maps of varieties, the technical work required to show existence of the limit is likely formidable. We intend to address such questions in future work.

\subsection*{Outline}
In Section 2, we introduce key well-known tools from the theory of dynamical degrees of rational maps.

In Section 3, we define lattice maps and their first dynamical degrees.

In Section 4, we begin to study skew pentagram maps, proving basic properties, such as dominance.

In Section 5, we give an upper bound on the dynamical degree of skew pentagram maps.

In Section 6, we study equal-length maps, proving the first part of Theorem \ref{thm_main}.

In Section 7, we study a truly skew map on rotationally symmetric octagons. We prove Theorem \ref{thm_no_fib} as Corollary \ref{cor_rs_consequences}, then prove the second part of Theorem \ref{thm_main} as Theorem \ref{thm_main_body}. Then we prove Corollary \ref{cor_height}.

\subsection*{Acknowledgments}
This work was supported by was supported by NSF Grant 2202752. The author thanks Yohsuke Matsuzawa, Jeffrey Diller, Niklas Affolter, Boris Khesin, and Anton Izosimov for helpful conversations related to this project. 

\section{Preliminaries} \label{sect_prelim}

\subsection{Dominant rational maps}
We work over an algebraically closed field $\bk$. Our convention is that a variety over $\bk$ is a reduced, separated scheme of finite type over $\bk$, not necessarily irreducible. A curve is a $1$-dimensional variety, and a surface is a $2$-dimensional variety. All varieties in this paper are quasi-projective. A rational map $f: X \dashrightarrow Y$ is the data of a Zariski dense open subset $U \subset X$ and a morphism $f_U: U \to Y$, up to the equivalence relation that $f = g$ whenever $f_U$ and $g_V$ agree on $U \cap V$. Each equivalence class $f$ has a representative $(U, f_U)$ that maximizes $U$ with respect to inclusion. The \emph{indeterminacy locus} of $f$, denoted $\Ind f$, is the complement of the maximal set $U$. If $f: X \dashrightarrow Y$ and $g: Y \dashrightarrow Z$, and for each component $V$ of $X$, the variety $f(V \smallsetminus \Ind f)$ is not contained in $\Ind g$, then we may define the composite $g \circ f$. In this situation, we say that $f$ and $g$ are composable. We say that $f$ and $g$ are \emph{inverses}, and write $f = g^{-1}$, if $g$ and $f$ are composable, $f$ and $g$ are composable, and
$$f \circ g = g \circ f = \Id.$$
We recursively define the \emph{iterates} of $f: X \dashrightarrow X$ by $f^0 = \Id$, and for all $m > 0$,
$$f^m = f \circ f^{m-1},$$
whenever these maps are composable. If all the iterates exist, we say $f$ is \emph{iterable}.

A rational map is called \emph{dominant} when its image is Zariski dense in the codomain. Dominant rational maps are composable, so they form a category.

If $f: X \dashrightarrow Y$ is a dominant rational map, and $Y$ is irreducible, there exists a Zariski open dense subset $U$ of $Y$ such that the cardinality of the fibers of $f|_{f^{-1}(U)}$ is constant. The value of this constant is the \emph{(generic) topological degree} of $f$, denoted $\deg_{\Top} f$.

Given $N \geq 1$, we denote the projective space of dimension $N$ over $\bk$ by $\PP^N$. Points in a projective space are written in homogeneous coordinates $[X_0 : \ldots : X_N]$. Projective spaces and their products are varieties, and our main setting of interest is rational maps between products of projective spaces.

Given $N_1, \ldots, N_r \geq 1$, endow $\PP^{N_i}$ with the homogeneous coordinates $X_0^{(i)}, \ldots, X_{N_i}^{(i)}$. Let 
$$\bk[X] = \bk[X_j^{(i)}: 0 \leq i \leq r, 0 \leq j \leq N_i].$$
An element of $\bk[X]$ is called \emph{multihomogeneous of multidegree $(d_0, \ldots, d_r)$} if, for each $0 \leq i \leq r$, it is homogeneous of degree $d_i$ in the coordinates $X_0^{(i)}, \ldots, X_{N_i}^{(i)}$. A tuple of multihomogeneous polynomials $(f_0, \ldots, f_M)$ in $\bk[X]$ of the same multidegree and defines a rational map
$$f : \PP^{N_1} \times \ldots \times \PP^{N_r} \dashrightarrow \PP^M,$$
$$(P_1, \ldots, P_r) \mapsto [f_0(P_1, \ldots, P_r) : \ldots : f_M(P_1, \ldots, P_r)]. $$
Every rational map $ \PP^{N_1} \times \ldots \times \PP^{N_r} \dashrightarrow \PP^M$ may be written in this form, and the defining tuple is unique up to simultaneous scaling of the components by elements of $\bk[X]$. So any such rational map $f$ has an associated tuple, unique up to scaling by units of $\bk$, in which the components $f_0, \ldots, f_M$ have no common factors of positive degree in $\bk[X]$. The \emph{multidegree} of $f$ is the common multidegree of $f_0, \ldots, f_M$ when written without common factors. 

The case of greatest interest for dynamics is when $r = 1$ and the domain and codomain are the same space $\PP^N$. In this situation, the multidegree of $f$ is just called the \emph{(algebraic) degree}, written $\deg f$. Note that $\deg f \neq \deg_{\Top} f$ in general. 

If $f$ is dominant, then $f$ is iterable. The next lemma gives a convenient criterion for checking dominance. Let $f = (f_0, \ldots f_N)$ be an $(N+1)$-tuple of homogeneous forms over $\bk$ in $X_0, \ldots, X_N$. Let $Df$ denote the matrix of partial derivatives
    \begin{equation}
        Df \colonequals \left( \frac{\partial f_i}{\partial X_j}\right)_{0 \leq i, j \leq N}. \label{eq_Df}
    \end{equation}
    
\begin{lemma} \label{lem_homog_jacobian}
    Let $f = (f_0, \ldots f_N)$ be an $(N+1)$-tuple of homogeneous forms over $\bk$ in $X_0, \ldots, X_N$, all of the same degree.
    If $Df \neq 0$, then $f$ defines a dominant rational map $\PP^N \dashrightarrow \PP^N$.
\end{lemma}
    
\begin{proof}
    The map $F: \A^{N+1} \to \A^{N+1}$ defined by $(f_0, \ldots, f_N)$ has Jacobian matrix $Df(X_0, \ldots, X_N)$ at $(X_0, \ldots X_N)$, so if $Df \neq 0$, there exists a point $(X_0, \ldots, X_N)$ for which the differential of $F$ is surjective on the tangent space at that point. This shows that $F$ is dominant, so its projectivization $f$ is dominant as well. 
\end{proof}

\begin{defn}[\cite{MR1704282}] \label{def_dd1_pn}
    Suppose that $f: \PP^N \dashrightarrow \PP^N$ is iterable. The \emph{dynamical degree} $\lambda_1(f)$ of $f$ is the exponential degree growth rate  
    \begin{equation*}
        \lambda_1(f) = \lim_{m\to \infty} (\deg f^m)^{1/m}.
    \end{equation*}
    The \emph{algebraic entropy} of $f$ is the logarithm of its dynamical degree,
    $$h_{\Alg}(f) = \log \lambda_1 (f).$$
\end{defn}

\subsection{Dynamical degrees for self-maps of projective varieties}

Let $X$ be a smooth projective variety. Given $0 \leq i \leq \dim X$, let $\Num^i X$ denote the group of algebraic $i$-cycles on $X$, that is, the free abelian group formally generated by irreducible subvarieties of codimension $i$. In the case $i = 1$, this specializes to the group $\Num X$ of Weil divisor classes up to numerical equivalence. 

Given a dominant rational map $f: X \dashrightarrow Y$, there is an induced \emph{pullback} homomorphism $f^{*,i} : \Num^i Y \to \Num^i X$. If $i = 1$, we suppress the notation, simply writing
$$f^* : \Num Y \to \Num X.$$
Roughly speaking, if $\Delta$ is a divisor on $Y$, then $f^* \Delta$ is just $f^{-1}(\Delta)$, taking multiplicity and indeterminacy into account. The pullback operation is not functorial in general, although it is for morphisms. See \cite{MR3617981} for details. In this paper, we never work with the definition of $f^{*,i}$ directly, as Lemma \ref{lemma_pn_pullbacks} and Lemma \ref{lem_compute_pback} give algorithms for computing them in the cases we require.

The following definition extends Definition \ref{def_dd1_pn} to arbitrary smooth projective varieties. Further generalizations to nonsmooth, nonprojective varieties are possible, but we do not require them.

\begin{defn} \label{def_dd_rat}
Let $X$ be a smooth projective variety over $\bk$, and let $f:X \dashrightarrow X$ be a dominant rational self-map of $X$. Let $0 \leq i \leq \dim X$. If $\Delta$ is an ample divisor on $X$, then the \emph{$i$-th dynamical degree of $f$} is the limit of intersection multiplicities
$$\lambda_i(f) \colonequals \lim_{m \to \infty} \left( (f^m)^{*,i} \Delta^i \cdot \Delta^{\dim X-i} \right)^{1/m}.$$
Existence of the limit and independence of choice of $\Delta$ are proved in \cite[Theorem 1]{dangsingular}. Compatibility with Definition \ref{def_dd1_pn} is immediate from taking $X = \PP^N$ and $\Delta$ to be the hyperplane class; then $(f^m)^* \Delta \cdot \Delta^{N-1} = \deg f^m$.
\end{defn}

The next lemma is a simple consequence of Definition \ref{def_dd_rat}.
\begin{lemma} \label{lem_deg_top}
    If $f: X \dashrightarrow X$ is a dominant rational self-map of a smooth irreducible projective variety, then
    $\lambda_0(f) = 1$. If the base field $\bk$ has characteristic $0$, then $\lambda_{\dim X}(f) = \deg_{\Top}(f)$.
\end{lemma}

The next series of lemmas justify the interpretation of dynamical degrees as (the exponential of) a kind of entropy. The first of these is a conjugacy invariance property.
\begin{lemma}[\cite{dangsingular}] \label{lem_bir_inv}
    Let $X,Y$ be smooth projective varieties. If $f: X \dashrightarrow X$ is dominant, and $g: Y \dashrightarrow X$ is birational, then for all $i$,
    $$\lambda_i(g \circ f \circ g^{-1}) = \lambda_i(f).$$
\end{lemma}

The map $g \circ f \circ g^{-1}$ is called a \emph{model} of $f$. Note that the domain of a model may differ from the domain of $f$.

Like other notions of entropy, dynamical degrees work well with semiconjugacies.
\begin{lemma}
\label{lemma_rel_dd}
Suppose that $X,Y$ are irreducible, smooth projective varieties and we have a commutative diagram of dominant rational maps
    \begin{center}
    \begin{tikzcd}
        X \arrow[r,dashed,"f"] \arrow[d,dashed,"\pi"] & X \arrow[d,dashed,"\pi"]  \\
        Y \arrow[r, dashed,"g"] & Y \\
    \end{tikzcd}.
    \end{center}
    Then 
    \begin{enumerate}
        \item for all $0 \leq i \leq \dim Y$, we have
    $$ \lambda_{i} (f) \geq \lambda_i (g).$$
    \label{it_lemma_rel_dd_skew_prod}
    \item If $\dim Y < \dim X$, then there exists $0 \leq i \leq \dim X - 1$ such that     $$\lambda_{i+1} (f) \mid \lambda_i (f).$$
    \label{it_lemma_rel_dd_divide}
    \end{enumerate}
\end{lemma}

\begin{proof}
These are simple cases of the product formula for relative dynamical degrees \cite{MR4048444}.
\end{proof}

It is often useful to study entropy on an invariant subsystem to understand entropy of a larger system. The following lemma justifies this principle for dynamical degrees.
\begin{lemma} 
\label{lem_subsystem}
Let $V \subset X$ be an inclusion of projective varieties. Let $f : X \dashrightarrow X$ be a dominant rational map. Assume $V \not\subseteq \Ind f$, so $f|_V : V \dashrightarrow V$ is defined, and assume that $f|_V$ is dominant. Then for all $0 \leq i \leq \dim V$,
$$\lambda_i (f|_V) \leq \lambda_i (f).$$
\end{lemma}
\begin{proof}
See \cite[Proposition 3.2]{MR4831038}.
\end{proof}

\section{Lattice maps and dynamical degrees} \label{sect_latt}

In this section, we define lattice maps, a general class of pentagram-like maps that are defined by local rules. Lattice maps combine aspects of cellular automata and rational maps. Many specific cases have been studied under the name of \emph{lattice equations} in the mathematical physics literature; see \cite{MR4906745, MR2392894, MR2481234,MR4827588} for a small sampling.

Since the introduction of the notion of algebraic entropy \cite{MR1704282}, the theory of degree growth has developed independently in two directions: rational maps of varieties, and lattice equations. Our goal in this section is to reconcile these theories, so that it is possible to consider dynamical degrees of not-necessarily-invertible maps on infinite-dimensional domains. The invariants we compute in Theorem \ref{thm_main} are dynamical degrees of lattice maps (Definition \ref{def_dd_latt}), but our proof works by comparing these dynamical degrees to dynamical degrees of rational maps, where there are more techniques at our disposal (Section \ref{sect_prelim}).

To fix ideas, we only consider the domain $(\PP^N)^G$, where $\PP^N$ is a projective space over an algebraically closed field $\bk$ and $G$ is an abelian group.

Given a set $A$ and a set $B$, let $B^A$ denote the set of all maps $A \to B$. It is convenient to view maps $A \to B$ as tuples of elements in $B$, indexed by $A$. We will write the map $a \mapsto b_a$ as $(b_a : a \in A)$, as $(b_a)_{a \in A}$, or simply as $(b_a)$ when the indexing set is clear.

We work with a fixed abelian group $G$. Given subsets $A_1, A_2 \subseteq G$, let $A_1 + A_2$ denote their Minkowski sum
$$A_1 + A_2 = \{a_1 + a_2 : a_1 \in A_1, a_2 \in A_2\}.$$

\begin{defn}
A \emph{neighborhood} is a finite non-empty subset of $G$. A \emph{local rule on $(\PP^N)^G$} is given by a neighborhood $S$ together with a rational map
$$\varphi : (\PP^N)^S \dashrightarrow \PP^N.$$
The \emph{lattice map}, or \emph{global map}, associated to the local rule $(S, \varphi)$ is the partially defined map
$$ T_\varphi : (\PP^N)^G \dashrightarrow (\PP^N)^G,$$
$$ (v_i)_{i \in G} \mapsto (\varphi( v_{i+s} : s \in S))_{i \in G}.$$
Maps of the form $T_\varphi$ are called \emph{lattice maps on $(\PP^N)^G$}. We say $\varphi$ \emph{defines} $T_{\varphi}$. When the local rule $\varphi$ is regular (so $\Ind \varphi = \varnothing$), we say the lattice map is \emph{regular} and write $T_\varphi: (\PP^N)^G \to (\PP^N)^G$.

More generally, given a non-empty subset $A \subseteq G$, there is a \emph{regional map} defined by
$$ \varphi_A : (\PP^N)^{A + S} \dashrightarrow (\PP^N)^A,$$
$$ (v_i)_{i \in G} \mapsto (\varphi( v_{i+s} : s \in S))_{i \in A}.$$
If $A$ is finite, then $\varphi_A$ is a rational map.
The set-theoretic domain of $\varphi_A$ is denoted $\Dom \varphi_A$ and consists precisely of those $v \in (\PP^N)^{A+S}$ such that, for all $i \in A$, we have $(v_{i + s} : s \in S) \not\in \Ind \varphi$.
\end{defn}

For each $j \in G$, there is a \emph{shift map} $\Sigma_j : (\PP^N)^G \to (\PP^N)^G$ defined by $(\Sigma_j v)_i = v_{i+j}$. Shift maps are lattice maps. We say that lattice maps $T, T'$ on $(\PP^N)^G$ \emph{agree up to a shift} if there exists $j \in G$ such that $T' = \Sigma_j \circ T$.

We say two local rules $(S_1,\varphi_1),(S_2,\varphi_2)$ on $(\PP^N)^G$ are \emph{equivalent} if they define the same lattice map. For example, the lattice map $x'_i = x_{i-1} + x_{i+1}$ on $(\A^1)^\Z$ could be defined with the neighborhood $\{-1,1\}$ or $\{-1,1,100\}$. Every equivalence class of local rules contains a unique representative minimizing the neighborhood.

We may define composition and iteration of lattice maps as we do with rational maps. Here, there is the usual technical difficulty having to do with domains of definition.

\begin{defn}
Let $T$ and $T'$ be lattice maps on $(\PP^N)^G$ defined by the local rules $(S,\varphi)$ and $(S',\varphi')$ respectively. Suppose that the rational maps $\phi_{S'}$ and $\varphi$ are composable, meaning that the image of $\phi_{S'}$ is Zariski dense in $(\PP^N)^S$. Then the \emph{composite} $T' \circ T$ is the lattice map defined by the local rule $\varphi' \circ \phi_{S'}$ on neighborhood $S + S'$. In this situation, we say that $T'$ and $T$ are composable.

We say that $T$ and $T'$ are \emph{inverses}, and write $T' = T^{-1}$, if $T$ and $T'$ are composable, $T'$ and $T$ are composable, and
$$T' \circ T = T \circ T' = \Id.$$
\end{defn}

As with composition of rational maps, the domain of the composite as a lattice map may be strictly larger than the domain as a naive composition of partially defined maps. It is easy to check that the composite is independent of the representatives of local rules $\varphi, \varphi'$, and that a sufficient condition for composability is $T(\Dom T) \cap \Dom T' \neq \varnothing$.

\begin{defn}
We recursively define the \emph{iterates} of $T$ by $T^0 = \Id$, and for all $m > 0$,
$$T^m = T \circ T^{m-1},$$
whenever these maps are composable. If all the iterates exist, we say $T$ is \emph{iterable}.
\end{defn}

\begin{defn} \label{def_deg_latt}
The \emph{degree} of a local rule $(S,\varphi)$ on $(\PP^N)^G$ is the degree of the rational map $\varphi$, i.e., the shared total degree of the multi-homogeneous polynomials defining its components when $\varphi$ is written without common factors. The \emph{degree} of a lattice map $T$ on $(\PP^N)^G$, denoted $\deg T$, is the degree of any local rule for $T$. 
\end{defn}
Local rules on $(\PP^N)^G$ are equivalent precisely when their homogeneous polynomial formulas agree; only the neighborhoods may differ. Since the degree of a local rule is constant in equivalence classes, the degree of a lattice map on $(\PP^N)^G$ is well-defined. 

\begin{defn} \label{def_dd_latt}
The \emph{degree sequence} of an iterable lattice map $T$ on $(\PP^N)^G$ is the sequence $\deg T^m$. The \emph{(first) dynamical degree} of an iterable lattice map $T$ is the exponential growth rate of the degree sequence,
\begin{equation} \label{eq_dd_latt}
    \lambda_1 (T) \colonequals \lim_{m \to \infty} (\deg T^m)^{1/m}.
\end{equation}
The \emph{algebraic entropy} of $T$ is defined as $\log \lambda_1(T)$.
\end{defn}

\begin{lemma} \label{lem_first_iterate}
 The limit in \eqref{eq_dd_latt} exists, and
\begin{equation} \label{eq_dd_latt_first_iterate}
    \lambda_1 (T) \leq \deg T.
\end{equation}
\end{lemma}

\begin{proof}
If $T$ and $T'$ are lattice maps, then $\deg (T' \circ T) \leq (\deg T')(\deg T)$. Applying this to the iterates of $T$ shows that the degree sequence is submultiplicative, so the limit in \eqref{eq_dd_latt} exists by Fekete's lemma and is at most $\deg T$.
\end{proof}

It is often useful to know that a ``general'' point may arise as the output of a lattice map. We formalize this with the following definition, extending the notion of dominant rational map.
\begin{defn} \label{def_dom}
An \emph{algebraic subset} of $(\PP^N)^G$ is the vanishing locus of a collection of multi-homogeneous polynomials in terms of the coordinates on the factors $\PP^N$. 
We say $T$ is \emph{dominant} if the set-theoretic image of $T$ is not contained in any proper algebraic subset. 
\end{defn}

\begin{prop}
    Dominant lattice maps on $(\PP^N)^G$ form a category. In particular, dominant lattice maps are iterable.
\end{prop} 
\begin{proof}
    Let $T$ and $T'$ be dominant lattice maps defined by the local rules $(S,\varphi)$ and $(S',\varphi')$ respectively. Since $T$ is dominant, so is $\phi_{S'}$. So $\varphi' \circ \phi_{S'}$ exists. It is a local rule for $T' \circ T$ with neighborhood $S + S'$. To check that $T' \circ T$ is dominant, let $F$ be a nontrivial multi-homogeneous polynomial. Let $W \subset G$ be the set of indices appearing in the definition of $W$; so $W$ is finite and nonempty. Since $T$ is dominant, so is $\phi_{W + S'}$. Since $T'$ is dominant, so is $\varphi'_W$. Then $\varphi'_W \circ \phi_{W + S}$ exists and is dominant; its image is therefore not contained in the vanishing locus of $F$ in $(\PP^N)^W$, which is a proper closed subset. So the image of $T' \circ T$ is not contained in the vanishing locus of $F$.
\end{proof}

Dynamical degrees of dominant rational maps are birational invariants. This also holds for lattice maps, although we must be careful about the meaning of ``birational.''

\begin{lemma}
    If $f, g$ are dominant lattice maps on $(\PP^N)^G$, and $g$ is invertible in the category of dominant lattice maps, then
    $$\lambda_1(f) = \lambda_1(g \circ f \circ g^{-1}).$$
\end{lemma}
\begin{proof}
    Let $h = g \circ f \circ g^{-1}$. For each $m \geq 0$,
    \begin{align*}
        \deg (h^m) &\leq \deg (g \circ f^m \circ g^{-1})\\
        &\leq (\deg g)(\deg f)^m (\deg g^{-1}) & \textrm{(by submultiplicativity)}.
    \end{align*}
    Taking $m$-th roots and passing to the limit gives
    $$\lambda_1(h) \leq \lambda_1(f).$$
    But the same argument applied to $h$ and $f = g^{-1} \circ h \circ g$ yields
    $$\lambda_1(f) \leq \lambda_1(h).$$
\end{proof}

If $G$ is finite, then lattice maps on $(\PP^N)^G$ are rational maps. The next proposition checks that our two definitions of first dynamical degree are consistent.

\begin{lemma} \label{lemma_pn_pullbacks}
    Let $N \geq 1$ and $r \geq 1$, and let $f: \PP^{N_1} \times \ldots \times \PP^{N_r} \dashrightarrow \PP^N$ be a dominant rational map expressed in lowest terms by $(f_0, \ldots, f_N)$, where each $f_i$ is multihomogeneous of multidegree $(d_1, \ldots, d_r)$. If $H$ is the hyperplane class on the codomain $\PP^N$, and $H_i$ is the pullback of the hyperplane class on the $i$-th factor $\PP^{N_i}$ to $\PP^{N_1} \times \ldots \times \PP^{N_r}$, then
    $$ f^* H = \sum_{i = 1}^r d_i H_i. $$
\end{lemma}

\begin{proof}
    This is a straightforward generalization of \cite[Theorem 2.10]{MR3617981}: the class in $\PP^{N_1} \times \ldots \times \PP^{N_r}$ of a hypersurface of multidegree $(d_1, \ldots, d_r)$ is $\sum_{i=1}^r d_i H_i$. The pullback of a general hyperplane representing $H$ is a hypersurface of multidegree $(d_1, \ldots, d_r)$. 
\end{proof}

\begin{prop}
    If $G$ is a finite group and $T : (\PP^N)^G \dashrightarrow (\PP^N)^G$ is a dominant lattice map, then the dynamical degree of $T$ as a lattice map (Definition \ref{def_dd_latt}) agrees with the dynamical degree of $T$ as a rational map (Definition \ref{def_dd_rat}).
\end{prop}

\begin{proof}
    Given $m \geq 0$, let $\varphi_m$ be a local rule for $T^m$ with neighborhood taken to be all of $G$. Let $n = \#G$. Then $\varphi_m: (\PP^N)^G \to \PP^N$ has some multidegree $(d_{m,g} : g \in G)$, and $\deg \varphi_m = \sum_{g \in G} d_{m,g}$. 
    
    Let $H$ denote the hyperplane class on $\PP^N$. For all $g \in G$, let $\pi_g : (\PP^N)^G \to \PP^N$ denote the projection to factor $g$, and let $H_g = \pi_g^* H$. Then by Lemma \ref{lemma_pn_pullbacks} and shift equivariance of $T^m$, we have
    $$ (T^m)^* H_g =  \sum_{j \in G} d_{m,j} H_{g+j}.$$
    Let $\Delta = \sum_{g \in G} H_g$. Then
    \begin{align*}
    (T^m)^* \Delta &= \sum_{g \in G} (T^m)^* H_g \\
    &=  \sum_{g \in G} \sum_{j \in G} d_{m,j} H_{g+j} \\
    &=  \sum_{h \in G} \left( \sum_{i \in G} d_{m,h-i} \right) H_{h} \\
    &= \sum_{h \in G} (\deg \varphi_m) H_{h} \\
    &= (\deg \varphi_m) \Delta.
    \end{align*}
    So 
    \begin{equation}
        (T^m)^* \Delta \cdot \Delta^{nN - 1} = \Delta^{nN} \deg \varphi_m.
        \label{eq_consistency_proof}
    \end{equation}
    By amplitude of $\Delta$, the top-level intersection multiplicity $\Delta^{nN}$ is positive, and it is independent of $m$. Since $\Delta$ is an ample class, taking $\lim_{m \to \infty} ((\;\cdot\;)^{1/m})$ in \eqref{eq_consistency_proof} yields Definition \ref{def_dd_rat} on the left, and Definition \ref{def_dd_latt} on the right.
\end{proof}

Suppose that $G_0$ is a quotient group of $G$. An element of $(\PP^N)^G$ is called \emph{$G_0$-periodic} if entries agree whenever their indices agree in $G_0$. The set of $G_0$-periodic elements is an algebraic subset of $(\PP^N)^G$. Suppose that $T = T_\varphi$ is a lattice map on $(\PP^N)^G$ and that $\Dom T$ includes at least one $G_0$-periodic element; this occurs, for instance, if $T$ is defined by a local rule on a neighborhood $S$ whose elements are distinct modulo $G_0$. Then, since $T$ commutes with shifts, it is semiconjugate to a lattice map
$$ T_0 : (\PP^N)^{G_0} \dashrightarrow (\PP^N)^{G_0}.$$
\begin{lemma} \label{lem_dd_periodic_comparison}
    Suppose that $T$ is an iterable lattice map on $(\PP^N)^G$, that $G_0$ is a quotient group of $G$, and that the induced iterable lattice map $T_0$ on $(\PP^N)^{G_0}$ exists. Then 
\begin{equation} 
\lambda_1 (T_0) \leq \lambda_1 (T).
\end{equation}
\end{lemma}
\begin{proof}    
A local rule for $T_0$ may be obtained by identifying indices in the formula for $\varphi$ modulo $G_0$. 
It follows that
\begin{equation} \label{eq_deg_periodic_comparison_termwise}
\deg T_0 \leq \deg T.
\end{equation}
Since $T$ and $T_0$ are both iterable, we may apply \eqref{eq_deg_periodic_comparison_termwise} to each term of the degree sequence.
\end{proof}

The following criterion is helpful in establishing dominance of lattice maps.
\begin{lemma} \label{lem_mod_criterion}
    Suppose that $T$ is a lattice map on $(\PP^N)^\Z$, and for infinitely many values of $n$, the induced lattice map on $(\PP^N)^{\Z/n\Z}$ exists and is dominant. Then $T$ is dominant.
\end{lemma}
\begin{proof}
Suppose by way of contradiction that the image of $T$ lies in some proper algebraic subset $A$. Let $F$ be a nontrivial multi-homogeneous polynomial that vanishes identically on $A$. Let $W \subset \Z$ be the set of indices appearing in $F$. Given $n$, let $F_0$ be the polynomial obtained from $F$ by reducing indices of variables modulo $n$. If $n$ is sufficiently large, the elements of $W$ are distinct modulo $n$, so $F_0$ is nontrivial. Therefore the vanishing locus of $F_0$ is a proper algebraic subset of $(\PP^N)^{\Z/n\Z}$ containing the image of the induced lattice map on $(\PP^N)^{\Z/n\Z}$.
\end{proof}

Verifying dominance of lattice maps may otherwise be difficult, as indicated by the following pair of examples.
\begin{example}
    Let $T$ be the lattice map on $(\A^2)^\Z$ defined by the local rule
    $$ \varphi((x_{0},y_{0}), (x_{1},y_1)) = (x_0,x_1).$$
    Then $\varphi$ is dominant, but $T$ is not.
\end{example}

\begin{example}
    Let $T$ be the lattice map on $(\A^1)^\Z$ defined by the local rule
    $$ \varphi(x_{0}, x_{1}) = x_0 - x_1.$$
    Then $T$ is dominant, but none of the induced lattice maps on $(\A^1)^{\Z/n\Z}$ are.
\end{example}

\section{Skew pentagram maps}

\subsection{Skew pentagram maps as lattice maps}
We work in the projective plane $\PP^2$ over an algebraically closed field $\bk$. A \emph{polygon} is a bi-infinite sequence $v \in (\PP^2)^\Z$. 
 The point $v_i$ in $\PP^2$ is the \emph{$i$-th vertex}.  
 The \emph{diagonals} of $v$ are the lines $\ovl{v_i v_j}$, for all pairs $i,j \in \Z$ such that $v_i \neq v_j$. If $v_i = v_j$, we say that the corresponding diagonal is undefined.
 We make no assumptions on the relative positions of the vertices, so degenerate configurations such as constant sequences are polygons.
 
A \emph{closed $n$-gon} is an element of $(\PP^2)^{n}$. We interpret closed $n$-gons as $n$-periodic infinite polygons, identifying $(\PP^2)^{\Z/n\Z}$ with $(\PP^2)^n$ via $(v_i)_{i \in \Z} \mapsto (v_1, \ldots, v_n)$.

Let $a,b,c,d \in \Z$ be distinct. The \emph{skew pentagram map} associated to the $4$-tuple $(a,b,c,d)$ is the lattice map
\begin{align}
    T_{a,b,c,d} & : (\PP^2)^\Z \dashrightarrow (\PP^2)^\Z 
\end{align}
with local rule $(S, \varphi)$ defined by
$$S = \{a,b,c,d\},$$
$$\varphi : (\PP^2)^S \dashrightarrow \PP^2,$$
$$(v_s : s \in S) \mapsto \ovl{v_a v_b} \cap \ovl{v_c v_d}.$$
For all $i \in \Z$, we have
\begin{equation}
    v'_i = \varphi(v_{i+s} : s \in S) = \ovl{v_{i+a} v_{i+b}} \cap \ovl{v_{i+c} v_{i+d}}, \label{eq_def_phi_skew}
\end{equation}
consistent with \eqref{eq_sk_intro}.

Note that $\Dom T_{a,b,c,d}$ includes all polygons such that, for all $i \in \Z$, the diagonals in \eqref{eq_def_phi_skew} are defined, and $\ovl{v_{i+a} v_{i+b}} \neq \ovl{v_{i+c} v_{i+d}}$. It follows that, if $n \geq 4$ and $a,b,c,d$ are distinct modulo $n$, then there is an induced lattice map on closed $n$-gons, 
$$T^{(n)}_{a,b,c,d}: (\PP^2)^n \dashrightarrow (\PP^2)^n,$$
which is also just a rational map.

The sets $\{a,b\}$ and $\{c,d\}$ are the \emph{defining diagonals} of $T_{a,b,c,d}$. The \emph{lengths} of the defining diagonals are $\abs{b-a}$ and $\abs{d-c}$, respectively.
We say that $T_{a,b,c,d}$ is \emph{equal-length} if $\abs{b-a} = \abs{d-c}$. Otherwise, we say that $T_{a,b,c,d}$ is \emph{truly skew}.

\begin{obs}
We required $a,b,c,d$ to be distinct to avoid trivialities. If any of these parameters coincide, the operation \eqref{eq_def_phi_skew} either involves undefined quantities or just produces a shift.
\end{obs}

\begin{obs} \label{obs_shift}
Composing by a shift allows us to reduce the study of individual skew pentagram maps to the case $a = 0$, so the family of skew pentagram maps is indexed by only $3$ parameters up to shifts.
\end{obs}

\begin{obs} \label{obs_Z_conj}
Up to renaming the parameters and conjugating by an automorphism of the underlying group $\Z$, every skew pentagram map is of the form $T_{a,b,c,d}$, where $a<b$, $a<c$, and $c<d$. So we may roughly sort the class of pentagram maps into three kinds, depending on whether $b < c$, $c < b < d$, or $d < b$; the last kind contains only truly skew maps. The direction-reversal involution
    $$\Phi : (\PP^2)^\Z \to (\PP^2)^\Z,$$
    $$(\Phi v)_i = v_{-i}$$ 
    conjugates $T_{a,b,c,d}$ to $T_{-a,-b,-c,-d}$.
So $\Phi$ conjugates an equal-length map to a shift of itself.
\end{obs}

\begin{obs} \label{obs_cart_power}
    On infinite polygons,
    if $k \in \N$, then $T_{ka,kb,kc,kd}$ acts separately on vertices in each residue class modulo $k$, so $T_{ka,kb,kc,kd}$ is conjugate to the $k$-fold Cartesian power of $T_{a,b,c,d}$. The main case of interest is therefore when $\gcd(a,b,c,d) = 1$. Similarly, on closed $n$-gons, the main case of interest is when $\gcd(n,a,b,c,d) = 1$. Otherwise, we can express $T^{(n)}_{a,b,c,d}$ as a Cartesian power of a simpler map.
\end{obs}

\begin{example}
    Every map $T^{(4)}_{a,b,c,d}$ on quadrilaterals is equal-length and non-iterable. Up to shifts and conjugacies, the only skew pentagram maps on closed polygons are the classical pentagram map $T^{(5)}_{0,2,1,3}$ and the truly skew map $T^{(5)}_{0,2,1,4}$. 
\end{example}

\subsection{Skew pentagram maps on moduli spaces}

The local rule \eqref{eq_def_phi_skew} only uses the projectively natural operations of forming and intersecting lines, so it respects the diagonal action of the group $\PGL_3$ of projective transformations. Therefore, skew pentagram maps also respect the diagonal action of $\PGL_3$. More precisely, given $A \in \PGL_3$ and $v = (v_i)_{i \in \Z} \in (\PP^2)^\Z$, let
$$A \cdot v = (Av_i)_{i \in \Z}.$$
If $v \in \Dom T_{a,b,c,d}$, then $A \cdot v \in \Dom T_{a,b,c,d}$, and
$$A \cdot (T_{a,b,c,d} (v)) = T_{a,b,c,d} (A \cdot v).$$
It follows that there is an induced partially-defined self-map of the set of projective equivalence classes of polygons, denoted 
$$\bar{T}_{a,b,c,d} : (\PP^2)^\Z/\PGL_3 \dashrightarrow (\PP^2)^\Z/\PGL_3.$$
Similarly, the skew pentagram map $T^{(n)}_{a,b,c,d}$ commutes with the diagonal action of $\PGL_3$ on $(\PP^2)^n$ given by
$$A \cdot (v_1, \ldots, v_n) = (Av_1, \ldots, Av_n).$$
It follows that there is an induced partially-defined map $$\bar{T}^{(n)}_{a,b,c,d}: (\PP^2)^n/\PGL_3 \dashrightarrow (\PP^2)^n/\PGL_3.$$
For the purposes of algebraic dynamics, we wish to view $\bar{T}^{(n)}_{a,b,c,d}$ as a rational map. However, its domain is just a topological space, not a variety. We remedy this in standard fashion, by throwing away certain ``bad'' points in the domain.

\begin{defn}
An $n$-tuple $(v_1, \ldots, v_n)$ of points in $\PP^2$ is in \emph{general linear position} if:
\begin{enumerate}
    \item for all distinct pairs $1 \leq i,j \leq n$, we have $v_i \neq v_j$, and
    \item for all distinct triples $1 \leq i,j,k \leq n$, the points $v_i, v_j, v_k$ are not collinear.
\end{enumerate}
The space of $n$-tuples in general linear position is denoted $(\PP^2)^n_\GLP$.
\end{defn}

A \emph{projective frame} is a $4$-tuple in $\PP^2$ of points in general linear position. This is the projective analogue of the notion of basis in linear algebra. Given any two projective frames $(v_1,v_2,v_3,v_4)$ and $(w_1,w_2,w_3,w_4)$, there is a unique projective transformation $A$ such that
$$ A \cdot (v_1,v_2,v_3,v_4) = (w_1,w_2,w_3,w_4). $$

Let $\mc{C}_n = (\PP^2)^n_\GLP / \PGL_3$ denote the moduli space of $n$-gons in general linear position up to $\PGL_3$-equivalence. Since $n \geq 4$, $\mc{C}_n$ is nonempty, exists as an algebraic variety, is a fine moduli space, and is a Zariski open subset of $\A^{n-8}$. Indeed, we can impose a normal form on $n$-gons in general linear position as follows. Select $4$ distinct indices $i_1, i_2, i_3, i_4$ and a preferred choice of projective frame. Then, given $v \in (\PP^2)^n_\GLP$, there is a unique projective transformation taking $(v_{i_1}, v_{i_2}, v_{i_3}, v_{i_4})$ to the preferred frame, leaving $n-4$ vertices that vary in $\PP^2$. It follows that
$$(\PP^2)^n_{\GLP} \cong \mc{C}_n \times \PGL_3.$$

If $T_{a,b,c,d}^{(n)}$ is dominant, then there is an induced dominant rational map $\bar{T}_{a,b,c,d}^{(n)}$ satisfying the commutative diagram
\begin{center}
\begin{tikzcd}
    (\PP^2)^{n} \arrow[r, "T_{a,b,c,d}^{(n)}",dashed] \arrow[d,dashed] & (\PP^2)^{n} \arrow[d,dashed] \\
    \mc{C}_n \arrow[r,"\bar{T}_{a,b,c,d}^{(n)}",dashed]  & \mc{C}_n
\end{tikzcd}.
\end{center}

\begin{remark}
There are larger moduli spaces birational to $\mc{C}_n$ that allow for some more degenerate polygons, defined using geometric invariant theory \cite{GIT} or corner invariants \cite{OST}. The choice of model does not matter for our purposes because dynamical degrees are birational invariants.
\end{remark}

\subsection{Iterability and dominance}

The next lemma assures us that the first dynamical degree of a skew pentagram map over $\C$ is meaningful.
\begin{lemma} \label{lem_iterable}
    Over $\C$, skew pentagram maps $T_{a,b,c,d}$ are iterable.
\end{lemma}

\begin{proof}
    It suffices to find a polygon that is carried to a projectively equivalent polygon by $T_{a,b,c,d}$. For sufficiently large $n$, a regular $n$-gon $v$ in $\R^2 \subset \C^2 \subset \PP^2_\C$ suffices. To see this, observe that if $n$ is sufficiently large, then $T_{a,b,c,d}$ is defined at $v$, and its image $T_{a,b,c,d}(v)$ is a dilatation of $v$.
\end{proof}

While iterability alone is enough to define dynamical degrees, we will later need to choose special orbits of skew pentagram maps that avoid various kinds of collapsing behavior. We want to know that, in some sense, generic orbits stay generic. We make this formal via the notion of dominant lattice map (Definition \ref{def_dom}).

\begin{prop} \label{prop_dom}
Skew pentagram maps $T_{a,b,c,d}$ are dominant over any algebraically closed field $\bk$.    
\end{prop}

\begin{cor}
Skew pentagram maps are iterable over any algebraically closed field $\bk$.
\end{cor}

\begin{remark}
    The maps $T^{(n)}_{a,b,c,d}$ on closed $n$-gons need not be dominant, or even iterable.
\end{remark}

The proof of Proposition \ref{prop_dom} comprises the rest of this section. We break the proof into cases.

\begin{lemma} \label{lem_dom_outer}
    If $a<c<d<b$, then $T_{a,b,c,d}$ is dominant.
\end{lemma}

\begin{proof}
     By composing with a shift, we reduce to the case $a=0$. By Definition \ref{def_dom}, it suffices to show that, for all $n$, the regional rational map
     $$\varphi_{\{1,\ldots,n\}} : (\PP^2)^{\{1,\ldots,n+b\}} \dashrightarrow (\PP^2)^{\{1,\ldots,n\} },$$
     $$(v_1,\ldots, v_{n+b}) \mapsto (\varphi(v_i, v_{i+b}, v_{i+c}, v_{i+d}))_{1 \leq i \leq n}$$
     is dominant. 
     
     We prove this by induction.
     Dominance of $\varphi_{\{1\}}$ is immediate from dominance of the local rule $\varphi$. Now assuming that $\varphi_{\{1,\ldots,n\}}$ is dominant, we will show that
     $\varphi_{\{0,1,\ldots,n\}}$ is dominant.

      If $w= (w_1,\ldots, w_n)$ is general, it is the image of a configuration $v \in (\PP^2)^{\{1,\ldots,n+b\}}$ by the inductive hypothesis.
      We claim that there is a Zariski dense subset $W \subseteq \PP^2$ depending on $w$ such that, for all $w_0 \in W$, the configuration $(w_0,w_1,\ldots, w_n)$ is in the image of $\varphi_{\{0,1, \ldots, n+b\}}$. So the image of $\varphi_{\{0,1, \ldots, n+b\}}$ is Zariski dense in $(\PP^2)^{\{0,1,\ldots,n\} }$.
      
      We achieve this by explicitly constructing an input $(u_0, \ldots, u_{n+b})$ with image $(w_0,\ldots, w_n)$, as follows. By our generality assumptions, the lines 
      $\ell_1 = \ovl{w_0 v_b}$, $\ell_2 = \ovl{w_0 v_d}$,
      exist.
      Suppose that $u \in (\PP^2)^{\{1,\ldots,n+b\}}$ agrees with $v$, except possibly in the $c$-th factor. Then by the assumption that $0<c<d<b$, we have that $\varphi_{\{1,\ldots,n\}}(u)$ agrees with $\varphi_{\{1,\ldots,n\}}(v)$ except possibly in the $c$-th factor. 
      
      Now we split into cases depending on the size of $n$. If $n < c$, choosing $u_0 \in \ell_1 \smallsetminus v_b$ and $u_c \in \ell_2 \smallsetminus v_d$ gives
      $$ \varphi_{\{0,1,\ldots,n\}}(u_0, \ldots, u_{n+b}) =(w_0,w_1,\ldots, w_n).$$
      
      If $n \geq c$, then the line $\ell_3 = \ovl{w_c v_{d+c}}$ exists because $v$ is in general linear position. If $u_0 \in \ell_1 \smallsetminus v_b$ and $u_c = \ell_2 \cap \ell_3$, then 
      \[
      \varphi_{\{0,1,\ldots,n\}}(u_0,u_1,\ldots,u_{n+b}) = (w_0,w_1,\ldots,w_n).
      \]
\end{proof}

We could argue similarly to the previous lemma for the other possible orderings of $a,b,c,d$, but the arguments become more delicate. Instead, we cover the other cases using a lemma of a different flavor, which will also be used later in our study of $8$-gons. The inspiration for the lemma is that the identity map is the (technically excluded) degenerate case of $T_{a,b,c,d}$ where $a = d = 0$. We show that, for some values of $(a,b,c,d)$, the differential of $T_{a,b,c,d}$ permutes a basis of tangent vectors at certain closed polygons.

\begin{lemma} \label{lem_dom_quadrilateral_trick}
Let $T_{0,b,c,d}$ be a skew pentagram map, let $n > 0$ be a multiple of $d$, and suppose that $a,b,c,d$ are distinct mod $n$.
Then $T_{0,b,c,d}^{(n)} \colon (\PP^2)^n \dashrightarrow (\PP^2)^n$ is dominant.
\end{lemma}

\begin{proof}
     By Observation \ref{obs_Z_conj}, we may conjugate to assume $0 < d$. It suffices to show that the differential of $T^{(n)}_{0,b,c,d}$ is injective on the tangent space at some $p \in \Dom T^{(n)}_{0,b,c,d}$. Let $p = (p_1,\ldots,p_{d})$ be a $d$-gon in general linear position. Consider the periodic extension of $p$ to $(\PP^2)^\Z$. By assumption on $d$, this polygon $p$ is a closed $n$-gon.
     
      We see by direct check that $p \in \Dom T^{(n)}_{0,b,c,d}$, since for all $i \in \Z/n\Z$, the diagonals $\ovl{p_{i} p_{i+b}}$ and $\ovl{p_{i+c} p_{i+d}}$ are defined and distinct, meeting at $p_{i}=p_{i+d}$.
     We may define a basis for the tangent space of $(\PP^2)^n$ at $p$ by choosing, for each $i$, a basis for the tangent space of the plane at $p_i$. Up to change of coordinates, we may assume that all the points $p_i$ are affine. Then we choose the basis $\alpha_i = \overrightarrow{p_i p_{i+c}}$, $\beta_i = \overrightarrow{p_i p_{i+b}}$.
     
     We claim that the differential of $T_{0,b,c,d}$ permutes the $\alpha_i$ and permutes the $\beta_i$. To see the first claim, suppose that $v$ is a general $n$-gon that agrees with $p$ except at the $i$-th vertex, such that $v_i \in \ovl{p_i p_{i+c}}$. Then $v' \colonequals T^{(n)}_{0,b,c,d}(v)$ agrees with $T_{0,b,c,d}^{(n)}(p)$ except possibly at vertices $i,i-b,i-c,i-d$. Checking these vertices,
     
\begin{align*}
    v'_{i} = v'_{i-a} &= \ovl{v_{i} v_{i+b-a}} \cap \ovl{v_{i+c-a} v_{i+d-a}}\\
     &= \ovl{v_{i} p_{i+b}} \cap \ovl{p_{i+c} p_{i}}\\
    &= v_i,\\
    \\
     v'_{i-b} &= \ovl{v_{i+a-b} v_{i}} \cap \ovl{v_{i+c-b} v_{i+d-b}}\\
     &= \ovl{p_{i-b} v_{i}} \cap \ovl{p_{i+c-b} p_{i-b}}\\
    &= p_{i-b},
    \end{align*}
    \begin{align*}
     v'_{i-c} &= \ovl{v_{i+a-c} v_{i+b-c}} \cap \ovl{v_{i} v_{i+d-c}}\\
        &= \ovl{p_{i-c} p_{i+b-c}} \cap \ovl{v_{i} p_{i-c}}\\
         &= p_{i-c},\\
             \\
    v'_{i-d} &= \ovl{v_{i+a-d} v_{i+b-d}} \cap \ovl{v_{i+c-d} v_{i}} \\
         &= \ovl{p_{i} p_{i+b}} \cap \ovl{p_{i+c} v_i} \\
         &= p_i.
\end{align*}

Since the output vertices except $v'_{i}$ are fixed and the function $v_i \mapsto v'_{i}$ is the identity, the differential of $T^{(n)}_{a,b,c,d}$ takes $\alpha_i$ to $\alpha_{i}$. A similar argument shows that the differential of $T^{(n)}_{a,b,c,d}$ takes $\beta_i$ to $\beta_{i-d}$. 
\end{proof}

\begin{proof}[Proof of Proposition \ref{prop_dom}]
Up to shifts and conjugacy, every skew pentagram map is of the form $T_{0,b,c,d}$ where $0<b$ and $0<c<d$. If $d<b$, then Lemma \ref{lem_dom_outer} applies. Otherwise, we have $b<d$, and any multiple $n > 0$ of $d$ satisfies the hypotheses of Lemma \ref{lem_dom_quadrilateral_trick}. Since an infinite sequence of induced lattice maps of increasing size are dominant, the map $T_{0,b,c,d}$ is also dominant by Lemma \ref{lem_mod_criterion}.
\end{proof}

\section{Upper bound}

In this short section, we prove a quick upper bound on the dynamical degree of a skew pentagram map using the first iterate estimate.

\begin{lemma} \label{lem_deg_phi}
    Every skew pentagram map satisfies $\deg T_{a,b,c,d} = 4$.
\end{lemma}
\begin{proof}
By Definition \ref{def_deg_latt}, we wish to show that $\deg \varphi = 4$, where $\varphi$ is the local rule for $T_{a,b,c,d}$. This claim is independent of the particular values of the parameters $a,b,c,d$, so to ease notation, take $(a,b,c,d)=(1,2,3,4)$.

Endow $\PP^2$ with the homogeneous coordinates $[X:Y:Z]$.
Let $(\PP^2)^\vee$ denote the dual projective plane of $\PP^2$. The points of $(\PP^2)^\vee$ are nontrivial homogeneous $1$-forms in $X,Y,Z$ up to scale, which we identify with the lines in $\PP^2$ where they vanish. Choose coordinates on $(\PP^2)^\vee$ so that the point $[\alpha:\beta:\gamma] \in (\PP^2)^\vee$ represents $\alpha X + \beta Y + \gamma Z$. 

Let
$$\ell: \PP^2 \times \PP^2 \dashrightarrow (\PP^2)^\vee,$$
$$(v,w) \mapsto \ovl{vw},$$
$$\wp : (\PP^2)^\vee \times (\PP^2)^\vee \dashrightarrow \PP^2,$$
$$(L_1, L_2) \mapsto L_1 \cap L_2.$$
In homogeneous coordinates, these two maps have the same formula:
\begin{equation} \label{eq_two_point_line}
\ell([X_1 : Y_1 : Z_1], [X_2 : Y_2 : Z_2]) = 
\begin{bmatrix}
    \begin{vmatrix}
        Y_1 & Y_2 \\
        Z_1 & Z_2 \\
    \end{vmatrix} :
    \begin{vmatrix}
        Z_1 & Z_2 \\
        X_1 & X_2 \\
    \end{vmatrix} :
    \begin{vmatrix}
        X_1 & X_2 \\
        Y_1 & Y_2 \\
    \end{vmatrix}
\end{bmatrix},
\end{equation}
\begin{equation} \label{eq_two_line_point}
\wp([\alpha_1 : \beta_1 : \gamma_1], [\alpha_2 : \beta_2 : \gamma_2]) = 
\begin{bmatrix}
    \begin{vmatrix}
        \beta_1 & \beta_2 \\
        \gamma_1 & \gamma_2 \\
    \end{vmatrix} :
    \begin{vmatrix}
        \gamma_1 & \gamma_2 \\
        \alpha_1 & \alpha_2 \\
    \end{vmatrix} :
    \begin{vmatrix}
        \alpha_1 & \alpha_2 \\
        \beta_1 & \beta_2 \\
    \end{vmatrix}
\end{bmatrix}.
\end{equation}
Then $\varphi(v_1,v_2,v_3,v_4)=\rho(\ell(v_1,v_2),\ell(v_3,v_4)).$ The components of the formula for $\varphi$ so obtained are homogeneous of degree $1$ separately in each set of coordinates $X_i,Y_i,Z_i$ for each $i \in \{1,2,3,4\}$. If there were a common factor $h$ of positive degree among the components, then, by symmetry, $h$ would be at least linear in each set of variables, so $\varphi$ would be a constant map, which it is not.
\end{proof}

\begin{prop} \label{prop_upper}
     Every skew pentagram map satisfies  $$\lambda_1(T_{a,b,c,d}) \leq 4.$$
\end{prop}

\begin{proof}
    By Proposition \ref{prop_dom}, the map $T_{a,b,c,d}$ is dominant, so $\lambda_1(T_{a,b,c,d})$ is defined.
    The first iterate estimate \eqref{eq_dd_latt_first_iterate} states that $\lambda_1(T_{a,b,c,d}) \leq \deg T_{a,b,c,d}$. By definition, we have $\deg T_{a,b,c,d} = \deg \varphi$, and Lemma \ref{lem_deg_phi} shows that $\deg \varphi = 4$.
\end{proof}

\section{Equal-length maps}
Recall that a skew pentagram map $T_{a,b,c,d}$ in the form $a < b$ and $c < d$ is called equal-length if $b-a = d-c$. Up to shifts and reflections, these maps comprise a $2$-parameter family $T_{0,b,c,b+c}$. Equal-length maps are known to have a cluster algebra interpretation \cite{MR3483131}, and they may be lifted to so-called $J$-corrugated maps in $\PP^3$, which admit Lax representations with spectral parameter from scaling symmetry \cite{izosimov}. The special case $T_{0,b,1,b+1}$ of intersecting consecutive diagonals has been studied extensively under the name \emph{deep-diagonal maps} \cite{MR4865926,MR3438382,MR4761767,textbookcasepentagramrigidity,MR3534837, MR3438382}.
Our point of view is that the general case of equal-length maps should be expected to behave just like  $T_{0,b,1,b+1}$. Indeed, on closed $n$-gons, if $\gcd(n,c) = 1$, then a relabeling of vertices conjugates $T_{0,b,c,b+c}^{(n)}$ to $T_{0,b',1,b'+1}^{(n)}$ for some $b'$. On infinite polygons, it is not possible to make this reduction, but the dynamics appear to be essentially the same.

In this section, we extend some integrability properties of deep-diagonal maps to equal-length maps, working in the language of lattice maps and dynamical degrees. These properties provide a counterpoint to our study of truly skew maps in Section \ref{sect_truly}.

\subsection{Invertibility}
Equal-length maps are invertible in the category of lattice maps, and their inverses are again equal-length maps. 

\begin{prop} \label{prop_inv_Z}
    Let $T_{a,b,c,d}$ be an equal-length map, written in the form $a<b$ and $c<d$. Then
\[
T_{a,b,c,d}^{-1} = T_{-d,-b,-c,-a} .
\]
\end{prop}

\begin{proof}
Let $T=T_{a,b,c,d}$, $T' = T_{-d,-b,-c,-a}$, $S = \{a,b,c,d\}$, $S' = \{-d,-b,-c,-a\}$, and let $\varphi$, $\varphi'$ be the local rules for $T$ and $T'$ on the neighborhoods $S$ and $S'$ respectively. We wish to show the following:
\begin{enumerate}
    \item $\varphi' \circ \phi_{S'}$ is defined as a rational map, and is equivalent to the identity map $\PP^2 \to \PP^2$ as a local rule;
    \item $\varphi \circ \varphi'_{S}$ is defined as a rational map, and is equivalent to the identity map $\PP^2 \to \PP^2$ as a local rule.
\end{enumerate}
By symmetry, it suffices to prove just (1). Since $T$ is dominant, so is $\varphi_{S'}$, so $\varphi' \circ \varphi_{S'}$ is defined as a rational map; and further, if $v \in (\PP^2)^{S + S'}$ is generic, then the four points of $\varphi_{S'} v$ are in general position, so the lines 
$$\ell_1 = \ovl{(\varphi_{S'} v)_{-a} ((\varphi_{S'} v)_{-c}}, \ell_2 = \ovl{((\varphi_{S'} v)_{-b} ((\varphi_{S'} v)_{-d}}$$
are defined with unique intersection point $\varphi'(\varphi_{S'} (v))$. We claim that $v_0 \in \ell_1 \cap \ell_2$, so $v_0 = \varphi'(\varphi_{S'} (v))$. To see this, observe that for each $s' \in S'$, we have
$$(\varphi_{S'} v)_{s'} \in \ovl{v_{s' + a} v_{s' + b}} \cap \ovl{v_{s' + c} v_{s' + d}}.$$
So
\begin{align*}
    &(\varphi_{S'} v)_{- a} \in \ovl{v_{0} v_{b-a}}, &(\varphi_{S'} v)_{-c} \in \ovl{v_{0} v_{d-c}},\\
    &(\varphi_{S'} v)_{- b} \in \ovl{v_{a-b} v_0}, &(\varphi_{S'} v)_{-d} \in \ovl{v_{c-d} v_0}.
\end{align*}
By the equal-length assumption, we have $b-a =  d - c$, so two distinct points of $\ell_1$ lie in $\ovl{v_0 v_{b-a}}$, and two distinct points of $\ell_2$ lie in $\ovl{v_{a-b} v_{0}}$.
So $v_0 \in \ell_1 \cap \ell_2$.
\end{proof}

\begin{remark}
Khesin-Soloviev generalized this invertibility property to $T_{I,J}$-maps in any dimension \cite[Theorem 1.4]{MR3438382} and universal pentagram maps \cite[Theorem 4.4]{MR3282373}, assuming dominance. The equal-length maps we work with here are precisely $T_{I,J}$-maps in the plane.
\end{remark}

\begin{remark}
In contrast, we speculate that all truly skew maps $T_{a,b,c,d}$ are non-invertible in the category of lattice maps, but proving this seems tricky due to indeterminacy. We do not have a clear notion of ``generic topological degree'' for lattice maps. We can, however, show that there exist non-invertible truly skew maps $T^{(n)}_{a,b,c,d}$ on closed $n$-gons, as these are rational maps  \ref{lem_basic} (\ref{it_lem_basic_top_deg}). This basic difference between the equal-length and truly skew cases is illustrated in Figure \ref{fig_invertibility}.
\end{remark}

\begin{figure}[b]
    \centering
    \includegraphics[width=0.7\linewidth]{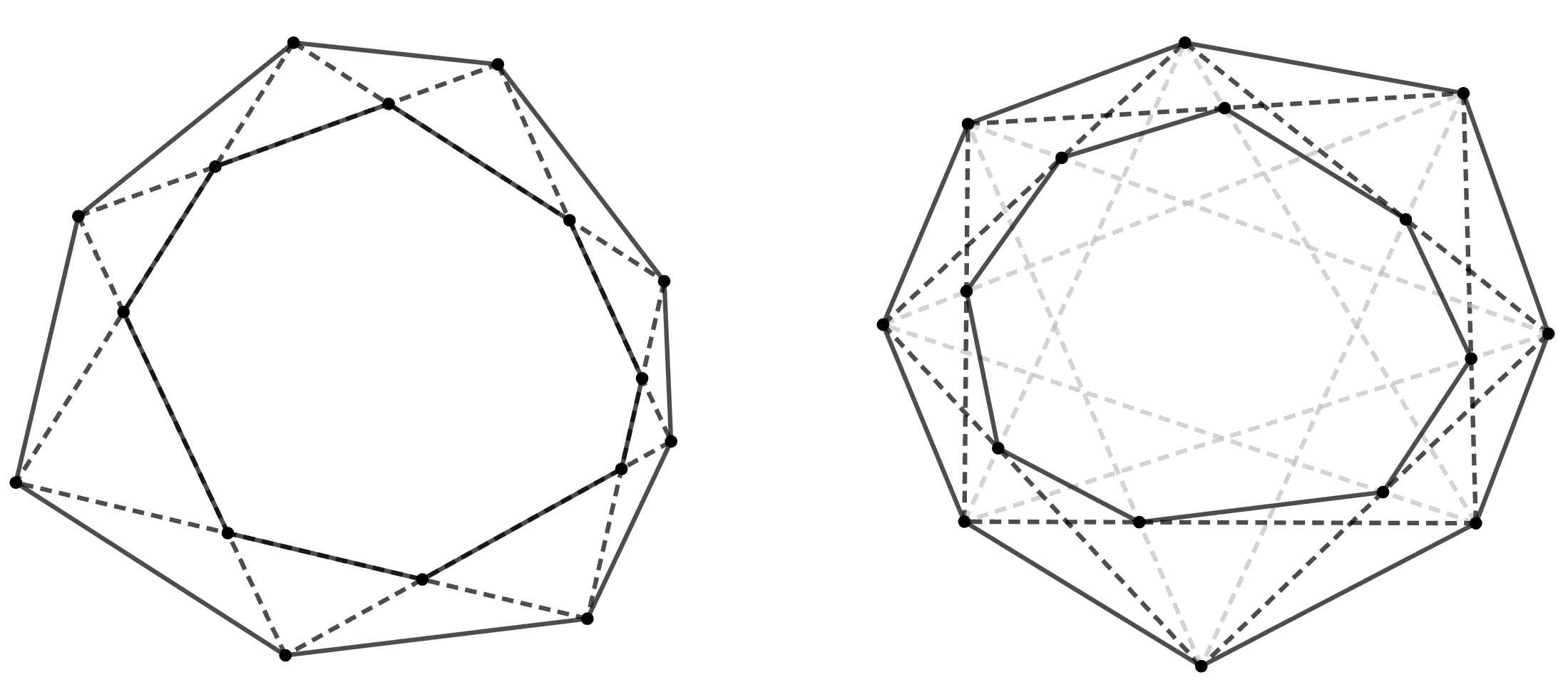}
    \caption{It is usually possible to reconstruct the input of $T_{0,2,1,3}$ from its output (left), but this is not the case for $T_{0,2,1,4}$ (right).}
    \label{fig_invertibility}
\end{figure}

\subsection{Degree growth}

Let $e_1,e_2,e_3$ be the standard basis vectors in $\Z^3$. 
The \emph{tetrahedral-octahedral lattice} is defined by
\[
\shL = \{ (i,j,k) \in \Z^3 : i + j + k \equiv 0 \mod 2\}.
\]
Let $\shL' = \Z^3 \smallsetminus \shL$. It is a copy of $\shL$, translated by any of the standard basis vectors.

\begin{defn}[\cite{MR4646096}]
    A map $w: \shL \to \PP^1_{\bk}$, written $r \mapsto w_r$, satisfies the \emph{Schwarzian octahedron recurrence}, or \emph{dSKP recurrence}, if in some affine chart of $\PP^1_\bk$ containing all the values of $w$, for all $r \in \shL'$, we have
\begin{equation} \label{eq_dskp}
  \frac{(w_{r-e_3}-w_{r+e_2})(w_{r-e_1}-w_{r+e_3})(w_{r-e_2}-w_{r+e_1})}{
  (w_{r+e_2}-w_{r-e_1})(w_{r+e_3}-w_{r-e_2})(w_{r+e_1}-w_{r-e_3})} = -1.
\end{equation}
\end{defn}
The ratio in \eqref{eq_dskp} is independent of the choice of the affine chart and is preserved by the diagonal action of $\PGL_2$.

\begin{example} \label{ex_dskp}
    For general $\mu_1, \mu_2, \mu_3 \in \bk$, the function
    $w_{(i,j,k)} = \mu_1 i + \mu_2 j + \mu_3 k$
    satisfies the dSKP recurrence.
\end{example}

Now we specialize to the base field $\C$. For each $k_0\in \Z$, let $\mc{I}_{k_0}$ be the subset of $\shL$ consisting of all $(i,j,k) \in \shL$ for which $k = k_0$. Treat the $w_s$ for all $s \in \mc{I}_0 \cup \mc{I}_1$ as formal indeterminates. Recursively applying \eqref{eq_dskp}, for each $r = (i,j,k) \in \shL$ with $k \geq 0$, there exists a rational function expressing $w_r$ in terms of these indeterminates, so
$$w_r = g_r (w_s : s \in \mc{I}_0 \cup \mc{I}_1).$$
Further, the degree $\delta_k \colonequals \deg g_{(i,j,k)}$ depends only on $k$.

Affolter-de Tili\`ere-Melotti studied the growth of the sequence $\delta_k$ and proved the following result \cite{MR4646096}.

\begin{thm} \label{thm_dskp}
    The sequence $\delta_k$ has at most polynomial growth in $k$.
\end{thm}

\begin{proof}
    Take lattice height function $h(i,j) = i + j \mod 2$ in \cite[Theorem 1.1]{MR4646096}. See also the immediately following discussion. Briefly, this initial data for the dSKP recurrence gives rise to explicit solutions in terms of partition functions. We note that this notion of height is unrelated to Weil height.
\end{proof}

It turns out that equal-length pentagram maps carry out the same computations as the dSKP recurrence, but rearranged. This insight is again due to Affolter-de Tili\`ere-Melotti, who studied corrugated pentagram maps, a cosmetically different variation \cite[Section 9.2 and 9.4]{MR4865926}. We follow their idea closely.

The link between the dSKP recurrence and pentagram maps comes from the following variant of Menelaus' Theorem. We state it for $x$-coordinates, but of course it is also true replacing $x$ with $y$ throughout.

\begin{lemma} \label{lem_mene}
    Let $B, C, E, F \in \A^2_\C$ be points in general linear position with distinct $x$-coordinates. Let $A = \ovl{BF} \cap \ovl{CE}$ and $D = \ovl{BC} \cap \ovl{EF}$. If $A, D \in \A^2_\C$, then
    \begin{equation}
        \frac{x(A) - x(F)}{x(F)-x(B)} \cdot \frac{x(B)-x(D)}{x(D)-x(C)} \cdot \frac{x(C)-x(E)}{x(E)-x(A)} = -1. \label{eq_menelaus}
    \end{equation} 
\end{lemma}

\begin{proof}
    The affineness condition allows us to define the $x$-coordinates of the six points. Each of the three factors in the left side of \eqref{eq_menelaus} is invariant under projective transformations whenever the moved points remain affine, since the three points in each factor are collinear by assumption. So it suffices to prove the lemma for any particular choice of projective frame $B,C,E,F$. Since $2 \neq 0$ in $\C$, there is a projective frame with distinct $x$-coordinates such that $x(C) = -x(B)$, $x(E) = -x(F)$. Then we have $x(A) = x(D) = 0$ and the result follows from then evaluating the left side of \eqref{eq_menelaus}.
\end{proof}

A $6$-tuple of points $(A,B,C,D,E,F)$ satisfying the hypotheses of Lemma \ref{lem_mene} is said to be an \emph{affine Menelaus configuration}.

\begin{thm} \label{thm_equal_length}
Let $T= T_{a,b,c,d}$ be an equal-length pentagram map over $\C$. The degree sequence of $T$ has at most polynomial growth, so $\lambda_1(T) = 1$.
\end{thm}

\begin{proof}
    By Observation \ref{obs_shift}, we may compose with a shift to reduce to the case $T = T_{0,b,c,b+c}$. By Observation \ref{obs_cart_power}, we may further reduce to the case $\gcd(b,c) = 1$.

Since $T$ is dominant by Proposition \ref{prop_dom}, for all $v \in (\PP^2)^\Z$ outside a countable union of proper algebraic subsets:
\begin{enumerate}
    \item for all $m$, the iterate $T^m(v)$ is defined,
    \item for all $m$ and all $i$, we have $T^m(v)_i \in \A^2$,
    \item for all $m$ and all $i$, the $4$-tuple $T^m(v)_i$, $T^m(v)_{i+b}$, $T^m(v)_{i+c}, T^m(v)_{i+b+c}$ is in general linear position and these $4$ vertices have distinct $x$-coordinates.
\end{enumerate}
For any $(i,m) \in \Z^2$, let $v_{i,m} = T^m(v)_i$ and $x_{i,m} = x(v_{i,m})$. 
Let
$$B = v_{i+c,m+1}, C = v_{i+b+c, m+1}, E = v_{i+b,m+1}, F = v_{i, m+1}, A = v_{i+b+c, m},D = v_{i,m+2}.$$
By definition of skew pentagram map, we have $D = \ovl{BC} \cap \ovl{EF}$. Since equal-length maps are invertible by equal-length maps (Proposition \ref{prop_inv_Z}), we have $A = \ovl{BF} \cap \ovl{CE}$.
So $(A,B,C,D,E,F)$ is an affine Menelaus configuration. Applying Lemma \ref{lem_mene} yields
    \begin{equation}
        \frac{x_{i+b+c,m} - {x_{i,m+1}}}{x_{i,m+1}-x_{i+c,m+1}} \cdot \frac{x_{i+c,m+1}-x_{i,m+2}}{x_{i,m+2}-x_{i+b+c,m+1}} \cdot \frac{x_{i+b+c,m+1}-x_{i+b,m+1}}{x_{i+b,m+1}-x_{i+b+c,m+2}} = -1. \label{eq_menelaus_penta}
    \end{equation} 
Let $\eta: \Lambda \to \Z^2$ be the linear map defined on a basis of $\Lambda$ by 
$$\eta(1,-1,0) = (b,0), \eta(-1,-1,0) = (c,0), \eta(0,-1,1) = (0,1).$$
It follows that $w_r = x_{\eta(r)}$ is a solution to the dSKP recurrence. Since $\gcd(b,c) = 1$, the map $\eta$ is surjective. So for every pair $(i,m)$, the rational function $g_m$ expresses each $x_{i,m}$ in terms of $\{x_{i',m'} : (i',m') \in \mc{I}\}$. Since the degree growth of $g_m$ in $m$ is polynomial (Theorem \ref{thm_dskp}), so is the degree growth of $x_{i,m}$ in $m$ in terms of the variables $\{x_{i',m'} : (i',m') \in \mc{I}\}$. Running the same argument for the $y$-coordinates, we have shown that there are rational formulas for the coordinates of the vertices of $T^m(v)$ with at most polynomial degree growth in the coordinates of $v, T(v)$. Since the coordinates of $T(v)$ are themselves fixed rational functions of the coordinates of $v$, we are done.
\end{proof}

\begin{cor} \label{cor_equal_len_dd}
    Let $T^{(n)}_{a,b,c,d}$ be an equal-length pentagram map on closed $n$-gons, assumed dominant. Let $\ovl{T}^{(n)}_{a,b,c,d}$ be the induced map on the moduli space of closed $n$-gons. Then 
    $$\lambda_1(T^{(n)}_{a,b,c,d}) = \lambda_1(\ovl{T}^{(n)}_{a,b,c,d}) = 1.$$
\end{cor}

\begin{proof}
Follows from Theorem \ref{thm_equal_length} ($\lambda_1(T_{a,b,c,d}) = 1$), Lemma \ref{lem_dd_periodic_comparison} (comparing dynamical degrees of induced lattice maps), and Lemma \ref{lemma_rel_dd} (\ref{it_lemma_rel_dd_skew_prod}) (comparing dynamical degrees of semiconjugate rational maps).
\end{proof}

\begin{remark} \label{remark_lax}
Equal-length pentagram maps admit a Lax form up to a finite semiconjugacy \cite{izo_khes}. Corollary \ref{cor_equal_len_dd} is thus consistent with a conjecture of Khesin-Soloviev that rational maps admitting Lax forms have arithmetic degree $1$, or equivalently, dynamical degree $1$ \cite[Conjecture 6.6]{khesinsoloviev}.
\end{remark}

\begin{remark}
        A rational map is called \emph{algebraically completely integrable} if some iterate is conjugate to a family of translations on a relative Picard or Jacobian variety. See e.g. \cite{weinreich2021}. Any algebraically completely integrable map has dynamical degree $1$. This follows from the fact that translations on individual abelian varieties have dynamical degree $1$ \cite{MR4363581} combined with the relative dynamical degree formula \cite{MR4048444}. So Theorem \ref{thm_equal_length} would be consistent with algebraic complete integrability (perhaps up to finite semiconjugacy) of equal-length maps on closed $n$-gons. 
\end{remark}

\begin{remark}
In our setup, \eqref{eq_dskp} does not define a lattice map on $(\PP^1)^{\mc{I}_0}$. The problem is that $w$-values are recursively determined in two steps, with offset domains identified with $\mc{I}_0$ and $\mc{I}_1$. However, it does define a lattice map on 
$$\omega : (\PP^2)^{\mc{I}_0} \dashrightarrow (\PP^2)^{\mc{I}_0}$$
by identifying $(w_r : r \in \mc{I}_0 \cup \mc{I}_1)$ with $([w_r : w_{r+e_1+e_3} : 1] : r \in \mc{I}_0)$. Applying $\omega$ shifts the second coordinate to the first, and computes the new second coordinate by $g_{(1,0,1)}$, with neighborhood $((0,0),(1,1),(1,-1),(2,0))$. Other identifications are possible as well.
\end{remark}

\section{A non-integrable skew pentagram map} \label{sect_truly}

We showed in Proposition \ref{prop_upper} that the dynamical degree $\lambda_1(T_{a,b,c,d})$ of any skew pentagram map over an algebraically closed field $\bk$ is at most $4$. In this section, we show that, for an infinite family of truly skew pentagram maps over $\C$, we achieve the upper bound $\lambda_1(T_{a,b,c,d}) = 4$. We do this by combining Proposition \ref{prop_upper} with an argument that $\lambda_1(T_{a,b,c,d}) \geq 4$. Our technique is to use non-integrability of a subsystem to obtain non-integrability of a larger system. The crucial trick is to study just one truly skew map $T_{0,2,1,4}$ on a moduli space $\mc{C}_{8,\RS}$ of rotationally symmetric octagons, which is a surface. The behavior of the restricted map $\bar{T}_{\RS}$ resembles the heat map on the moduli space of closed $5$-gons \cite{KR}.

The reader may wonder why rotationally symmetric octagons are a natural choice to study. Our goal was to find an example of a skew pentagram map that is ``as chaotic as possible'' -- maximal dynamical degree and no preserved fibration. Experimenting with smaller polygons, we did not find any skew pentagram maps satisfying both requirements. As the number of vertices increases, the computations get more difficult, but octagons with $4$-fold rotational symmetry have an invariant subspace of dimension $2$ and a convenient normal form, making them a natural next candidate to check. 

We note that the equal-length map $T_{0,3,1,4}$ on $\mc{C}_{8,\RS}$ was proved to be completely integrable by Schwartz \cite{MR4761767,textbookcasepentagramrigidity}.

\subsection{Computing dynamical degrees on surfaces}

In this preliminary section, we recall some standard techniques for computing pullbacks and dynamical degrees of surface maps. These techniques are valid over any algebraically closed field. A general reference for this section is \cite{KR}.

The following lemma describes how to compute pullbacks.

\begin{lemma} \label{lem_compute_pback}
    Let $f: X \dashrightarrow Y$ be a dominant rational map between smooth projective surfaces. Suppose $C \subset Y$ is an irreducible curve defined by a local equation $\psi = 0$, where $\psi$ vanishes to order $1$ along $C$. Let $f^{(-1)}(C)$ denote the set $\ovl{(f|_{X \smallsetminus \Ind f}})^{-1} (C)$. Then
    \begin{equation}
        f^* C = \sum_{\substack{B \subset f^{(-1)}(C) \\ \textrm{irreducible}}} m_B B,
    \end{equation}
    where $m_B$ is the multiplicity of vanishing of $\psi \circ f$ along $B$.
\end{lemma}

\begin{proof}
    See \cite[Lemma 2.1]{KR}.
\end{proof}

The difficulty in computing dynamical degrees is that pullback is not functorial, so each homomorphism $(f^m)^*$ must be computed separately in general. The next definition characterizes maps for which iteration works well with pullback. Then we explain how to use this to compute dynamical degrees.

\begin{defn}
Let $f: X \dashrightarrow X$ be a dominant rational self-map of a smooth projective variety. Then we say that $f$ is \emph{algebraically stable} if, for all $m \geq 1$,
$$(f^m)^* = (f^*)^m.$$
\end{defn}

The \emph{spectral radius} of a linear endomorphism $T$ is $\rad T \colonequals \max_\lambda \abs{\lambda},$ where $\lambda$ ranges over the eigenvalues of $T$.

\begin{lemma} \label{lem_norms_and_rads}
    Let $f: X \dashrightarrow X$ be a dominant rational self-map of a smooth projective variety. If $f$ is algebraically stable, then 
$$\lambda_1(f) = \rad (f^*: \Num X \to \Num X).$$
\end{lemma}

Because of Lemma \ref{lem_norms_and_rads}, it is usually possible to explicitly compute dynamical degrees for algebraically stable maps. We therefore desire a geometric criterion for showing that a map is algebraically stable. We only require this in the surface case, where the ideas are straightforward.

A \emph{contracted curve} of a dominant rational map $f: X \dashrightarrow Y$ is an irreducible curve $C \subset \PP^2$ such that $f(C \smallsetminus \Ind f)$ is a point.

\begin{lemma} \label{lem_as_no_ctrcted}
    If $f: X \dashrightarrow X$ is a dominant rational self-map of a smooth projective surface and $f$ has no contracted curves, then $f$ is algebraically stable.
\end{lemma}

\begin{proof}
If $X,Y,Z$ are smooth surfaces and $f:X \dashrightarrow Y$ and $g: Y \dashrightarrow Z$ are dominant, then $g^* f^* = (g \circ f)^*$ if and only there if is no curve $C$ in $X$ for which $f(C \smallsetminus \Ind f) \subset \Ind g$; see \cite[Proposition 1.4]{MR3342248}. Since $\Ind g$ is $0$-dimensional, the latter condition holds if $f$ has no contracted curves.
\end{proof}

To locate all contracted curves, we generally need to work in charts. However, the following lemma shows that a single computation suffices for self-maps of $\PP^2$.

\begin{lemma} \label{lem_contracted_vs_det_DF}
If $C$ is a contracted curve of a dominant map $f: \PP^2 \dashrightarrow \PP^2$, then $C$ is contained in the vanishing locus of $\det Df$, where $Df$ is the matrix of partial derivatives in \eqref{eq_Df}.
\end{lemma}

\begin{proof}
In this proof, $\rmd$ denotes the differential, meaning the induced map on tangent bundles. Let $\pi: \A^3 \smallsetminus 0 \to \PP^2$ denote the projectivization map. At any point $y \in \A^3 \smallsetminus 0$, $\ker (\rmd \pi)_y$ consists of multiples of the vector of coordinates of $y$, and hence is $1$-dimensional. Further, if $F: \A^3 \smallsetminus 0 \to \A^3 \smallsetminus 0$ is a homogeneous map of degree $d \geq 1$, then since $F$ commutes with scaling, we have $(\rmd F)(\ker (\rmd \pi)_y) \subset \ker (\rmd \pi)_{F(y)}$. Since the vanishing locus of $\det Df$ is Zariski closed, it suffices to show that if $x \in C \smallsetminus \Ind f$, then $\det Df(x) = 0$. If $x \in C \smallsetminus \Ind f$, then any Zariski tangent vector to $C$ is contained in $\ker \rmd f$. Let $F$ be the affine lift of $f$, so $\pi \circ F = f \circ \pi$. Let $y \in \pi^{-1}(x)$. Let $\vec{v}$ be a tangent vector to $y$ such that $(d \pi)(\vec{v})$ is a nontrivial tangent vector to $C$; this exists because $d \pi$ is surjective. Since $C$ is contracted, we have $(\rmd f) (\rmd \pi) \vec{v} = 0$, which implies that $(\rmd \pi)(\rmd F) \vec{v} = 0$ by the chain rule. It follows that $(\rmd F)^{-1}(\ker \rmd \pi)_{F(y)}$ is at least $2$-dimensional, since it contains both $\ker (\rmd \pi)_y$ and $\vec{v}$, which lies outside $\ker (\rmd \pi)_y$. So $(\rmd F)_y$ is not injective. Since $Df$ is the matrix of $\rmd F$ up to scale, we obtain $(\det Df)(y) = 0$, so $\det Df(x) = 0$.
\end{proof}

\subsection{Rotationally $4$-symmetric octagons}

An $8$-gon $v$ is \emph{rotationally $4$-symmetric} if there exists $R \in \PGL_3$ such that $v_{i+2} = Rv_i$ for all $0 \leq i < 8$.

Let $(\PP^2)^8_{\RS}$ denote the subset of $(\PP^2)^8_{\GLP}$ consisting of rotationally $4$-symmetric $8$-gons. Let $\mc{C}_{8,\RS}$ be the image of $(\PP^2)^8_{\RS}$ in the moduli space $\mc{C}_{8}$, giving the diagram
\begin{center}
\begin{tikzcd}
(\PP^2)^8_\RS \arrow[d] \arrow[r,hook] & (\PP^2)^8_{\GLP} \arrow[d] \\
 \mc{C}_{8,\RS} \arrow[r,hook] &\mc{C}_{8}.
\end{tikzcd}
\end{center}
We embed $\mc{C}_{8,\RS}$ as a Zariski dense open subset of $\A^2$ using the following normal form. The \emph{standard rotation} is
$$ R([X:Y:Z]) =[-Y:X:Z].$$
The line $L_\infty : \{Z = 0\}$ in $\PP^2$ is preserved by $R$. We call $L_\infty$ the \emph{line at infinity} and identify $\A^2 \cong \PP^2 \smallsetminus L_\infty$. This distinguished $\A^2$ is the \emph{main affine patch}. The $8$-gons in $(\A^2)^8$ are called \emph{affine}. Set coordinates $x=X/Z, y=Y/Z$ on $\A^2$. An affine closed $8$-gon is \emph{normalized} if
$$v_1 = (1,0), \; v_3 = (0,1), \; v_5 = (-1,0), \; v_7 = (0,-1).$$
Each element of $\mc{C}_{8}$ has a unique normalized representative.
If $v$ is both normalized and rotationally $4$-symmetric, then $v$ is of the form
$$v_2 = (x,y), \; v_4 = (-y,x), \; v_5 = (-x,-y), \; v_7 = (y,-x)$$
for some $(x,y) \in \A^2$. So we have an embedding $$\mc{C}_{8,\RS} \hookrightarrow \A^2,$$
$$(\text{class of } v ) \mapsto (v_2 \text{ of the normalized representative}).$$

Composing with the distinguished inclusion $\A^2 \hookrightarrow \PP^2$ gives an embedding
$$\iota: \mc{C}_{8,\RS} \hookrightarrow \PP^2.$$

\subsection{A truly skew pentagram map}

The specific map we analyze is $T = T^{(8)}_{0,2,1,4}$. We recall that $T$ is defined by
$$T: (\PP^2)^8 \dashrightarrow (\PP^2)^8,$$
$$ (v_i) \mapsto (v'_i),$$
\begin{equation}
    \ovl{v_{i} v_{i+2}} \cap \ovl{v_{i+1} v_{i+4}}. \label{eq_def_skew}
\end{equation}

\begin{prop}
    The restriction of $T$ to $(\PP^2)^8_{\RS}$ defines a rational self-map of $(\PP^2)^8_{\RS}$, denoted $T_{\RS}$. The restriction of $\ovl{T}$ to $\mc{C}_{8,\RS}$ is a rational map, denoted $\ovl{T}_{\RS}$. These form a diagram
\begin{center}
\begin{tikzcd}
(\PP^2)^8_\RS \arrow[d] \arrow[r,dashed,"T_{\RS}"] & (\PP^2)^8_{\RS} \arrow[d] \\
 \mc{C}_{8,\RS} \arrow[r,dashed,"\ovl{T}_{\RS}"] &\mc{C}_{8,\RS}.
\end{tikzcd}
\end{center}
\end{prop}

\begin{proof}
    Concerning $T_{\RS}$, there are two things to check. First, if $v \in (\PP^2)^8_\RS$, then $Tv$ is rotationally symmetric; this is because $T$ commutes with projective transformations. Second, if $v$ is general in $(\PP^2)^8_\RS$, then $Tv$ is in general linear position. To see this, notice that $T$ dilatates the normalized regular $8$-gon in $\R^2$. Concerning $\ovl{T}_{\RS}$, because the rational maps $T$ and $T_{\RS}$ commute with projective transformations, there are induced rational self-maps on the moduli spaces. The induced map on the quotient $\mc{C}_{8,\RS}$ necessarily agrees with $\ovl{T}$, so $\ovl{T}_{\RS}$ is the quotient map.
\end{proof}

Let $f: \PP^2 \dashrightarrow \PP^2$ be the projectivization of $\ovl{T}_{\RS}$ via $\iota$, forming the diagram
\begin{center}
\begin{tikzcd}
\mc{C}_{8,\RS} \arrow[d,hook, "\iota"] \arrow[r,dashed,"\ovl{T}_{\RS}"] & \mc{C}_{8,\RS} \arrow[d, hook, "\iota"] \\
 \PP^2 \arrow[r,dashed,"f"] & \PP^2.
\end{tikzcd}
\end{center}

\begin{prop} \label{prop_formula}
The map $f:\PP^2 \dashrightarrow \PP^2$ is given in homogeneous coordinates by the formula
     $$f([X:Y:Z]) = [f_0([X:Y:Z]), f_1([X:Y:Z]), f_2([X:Y:Z])],$$
     where
         \begin{align*}
        f_0 &= (X + Y + Z) ( (X - Y + Z)(-XYZ - X^2 Z + X^3 + X Y^2 ) \\ 
        & \hspace{0.4in} + 2 Y (X^2 Y + 2 X^2 Z + Y^3 + Y^2 Z - XYZ) ),
        \\
        f_1 &= (X + Y + Z) ( (-2Y)(-XYZ - X^2 Z + X^3 + X Y^2 ) \\ 
        & \hspace{0.4in} + (X - Y + Z)(X^2 Y + 2 X^2 Z + Y^3 + Y^2 Z - XYZ) ),
        \\
        f_2 &= (X - (1 + 2i)Y + Z) (X - (1 - 2i)Y + Z)(X^2 + Y^2 + X Z + Y Z) Z.
    \end{align*}
\end{prop}

\begin{proof}
Let $v \in \mc{C}_{8,\RS}$ be general, identify $v$ with its normalized representative, and let $v' = Tv$. Since $v'$ is in general linear position, there is a unique projective transformation $M$ such that $Mv'$ is normalized.
Let us calculate the formula of the rational map $f$ that takes $v_2$ to the normalized second vertex of $v'$.

By explicit calculation using the formula \eqref{eq_two_point_line}, we obtain
\begin{align*}
    v'_1 = [& X -Y + Z : 2Y : X + Y + Z],\\
    v'_2 = [& -XYZ - X^2 Z + X^3 + XY^2 : X^2 Y + 2 X^2 Z + Y^3 + Y^2 Z - XYZ \\
    &: X^2 Z + XZ^2 + Y^2 Z + YZ^2].
\end{align*}
Notice that $M$ and the standard rotation $R$ commute, since they are projective transformations with the same action on the $4$-tuple $(v'_1, v'_3, v'_5, v'_7)$. It follows that the matrix of $M$ is of the form
    \[
    \begin{pmatrix}
        P_1 & -P_2 & 0 \\
        P_2 & P_1 & 0 \\
        0 & 0 & P_3 \\
    \end{pmatrix}
    \]
    where $P_1, P_2, P_3$ are homogeneous polynomials in $X, Y, Z$. 
    Solving the system of equations $Mv'_1 = v_1$ for $P_1, P_2, P_3$, we get
    \begin{align*}
        P_1 &= (X + Y + Z)(X - Y + Z),\\
        P_2 &= -2Y(X + Y + Z),\\
        P_3 &= (X - (1 + 2i)Y + Z)(X - (1 - 2i)Y + Z).
    \end{align*}
\end{proof}

The next statement collects some basic geometric information about $f$ that can be assessed using a computer algebra system. Supporting calculations in SageMath \cite{sage} are available at \cite{code}.

We give names to the three fixed points of the standard rotation $R$:
$$p_1 = [0 : 0 : 1], \quad p_2 = [1 : i : 0], \quad p_3 = [1 : - i : 0].$$ 
\begin{lemma} \label{lem_basic} \leavevmode
\begin{enumerate}
    \item The rational map $f$ is dominant. \label{it_lem_basic_non_dominant}
    \item The rational map $f$ has algebraic degree $5$. \label{it_lem_basic_alg_deg}
    \item The contracted curves of $f$ are
    \begin{align*}
        C_1 :& \hspace{0.2in} X + Y + Z = 0,\\
        C_2 :& \hspace{0.2in} X - (1 + 2i)Y + Z = 0,\\
        C_3 :& \hspace{0.2in} X - (1 - 2i)Y + Z = 0.
    \end{align*}
    We have 
    $$f(C_1 \smallsetminus \Ind f) = p_1,\quad f(C_2 \smallsetminus \Ind f) = p_2, \quad f(C_3 \smallsetminus \Ind f) = p_3.$$ \label{it_lem_basic_contracted}
    \item There are $12$ indeterminacy points of $f$. These include $p_1, p_2, p_3$, so $f$ is not algebraically stable.
    \item The generic topological degree of $f$ is $5$. \label{it_lem_basic_top_deg}
\end{enumerate}
\end{lemma} 

\begin{proof}
\par(1)\enspace Applying Lemma \ref{lem_homog_jacobian}, since $\det Df \neq 0$, the map $f$ is dominant. Alternatively, we may use the following geometric argument, which we will also use in the proof of (3). The image of $f$ is irreducible, so it suffices to show that the image contains a curve and a point not on that curve. Let $L$ be the line defined by $X + Y = Z$. General points in $L$ correspond to normalized $8$-gons where $v_2 \in \ovl{v_1 v_3}$. Let $v' = Tv$. Computing directly from the definition of $T$ shows that $v'_2 \in \ovl{v'_1 v'_3}$, so $f(L \smallsetminus \Ind f) \subset L$. To see that this image is nonconstant, observe that the normalized $8$-gon $v$ with $v_1 = v_2$ is in the domain of $T$ and $v'_2 = v'_3$, and similarly, the normalized $8$-gon with $v_2 = v_3$ has $v'_1 = v'_2$. It follows that $f$ has a $2$-cycle
\begin{equation}
[1:0:1] \xmapsto{f} [0:1:1] \xmapsto{f} [1:0:1]. \label{eq_two_cycle}
\end{equation}
These points are in $L$, so $f$ generically maps $L$ onto $L$. The regular $8$-gon is fixed up to scale by $T$, so
$$[\sqrt{2}:\sqrt{2}:2] \xmapsto{f} [\sqrt{2}:\sqrt{2}:2],$$
providing a point outside $L$ in the image of $f$.

\par(2)\enspace
Proposition \ref{prop_formula} expresses $f$ with homogeneous polynomials $f_0,f_1,f_2$ of degree $5$. These polynomials have no simultaneous common factor, so $\deg f = 5$.

\par(3)\enspace 
If $C$ is a contracted curve of $f$, then $\det Df$ vanishes along $C$. So the contracted curves are among the irreducible components of the zero locus of $\det Df$. With computer algebra, we find the irreducible decomposition
$$\{ \det Df = 0\} = C_1 \cup C_2 \cup C_3 \cup C_4,$$
where $C_1, C_2, C_3$ are as written in the theorem statement and 
{\small
\begin{align*}
C_4 : 0 =  &-X^{9} - 3 X^{8} Y - 2 X^{7} Y^{2} - 6 X^{6} Y^{3} + 2 X^{3} Y^{6} + 6 X^{2} Y^{7} + X Y^{8} \\
&+ 3 Y^{9} + 2 X^{8} Z - 22 X^{7} Y Z - 4 X^{6} Y^{2} Z - 58 X^{5} Y^{3} Z - 20 X^{4} Y^{4} Z \\
&- 50 X^{3} Y^{5} Z - 20 X^{2} Y^{6} Z - 14 X Y^{7} Z - 6 Y^{8} Z + 13 X^{7} Z^{2} - 25 X^{6} Y Z^{2} \\
&+ 9 X^{5} Y^{2} Z^{2} - 93 X^{4} Y^{3} Z^{2} - 25 X^{3} Y^{4} Z^{2} - 83 X^{2} Y^{5} Z^{2} \\
&- 21 X Y^{6} Z^{2} - 15 Y^{7} Z^{2} + 4 X^{6} Z^{3} - 12 X^{5} Y Z^{3} - 36 X^{4} Y^{2} Z^{3} \\
&+ 16 X^{3} Y^{3} Z^{3} - 36 X^{2} Y^{4} Z^{3} - 12 X Y^{5} Z^{3} + 4 Y^{6} Z^{3} - 15 X^{5} Z^{4} \\
&- 21 X^{4} Y Z^{4} - 68 X^{3} Y^{2} Z^{4} - 4 X^{2} Y^{3} Z^{4} - 25 X Y^{4} Z^{4} + 13 Y^{5} Z^{4}\\
&- 6 X^{4} Z^{5} - 14 X^{3} Y Z^{5} - 8 X^{2} Y^{2} Z^{5} - 22 X Y^{3} Z^{5} + 2 Y^{4} Z^{5} \\
&+ 3 X^{3} Z^{6} + X^{2} Y Z^{6} - 3 X Y^{2} Z^{6} - Y^{3} Z^{6}.
\end{align*}}
The curve $C_4$ is not contracted, because $C_4$ contains $[1:0:1]$ and $[0:1:1]$, which have distinct images by \eqref{eq_two_cycle}. This limits the possible contracted curves to $C_1, C_2, C_3$.

To see that the lines $C_1,C_2,C_3$ are indeed contracted and have the claimed images, we may parametrize each and apply $f$. For instance, to see that $f(C_2 \smallsetminus \Ind f) = [1:i:0]$, we may parameterize $C_2$ by $r([Y:Z]) = [(1+2i)Y - Z:Y:Z]$ and compute directly that $f \circ r$ is the constant map with image $[1:i:0]$. 

\par \enspace (4) The indeterminacy locus is the locus where $f_1,f_2,f_3$ simultaneously vanish. Calculating this locus is a routine computer calculation.

\par \enspace (5) We apply a standard method, following the calculation of the topological degree of the heat map on pentagons by Schwartz \cite[Chapter 8]{MR3642552}.

The generic topological degree $\deg_{\Top} (f)$ is the maximum cardinality of the set $(f|_{\PP^2 \smallsetminus \Ind f})^{-1}(p)$ over all $p \in \PP^2 \smallsetminus \{p_1, p_2, p_3\}$. To compute $(f|_{\PP^2 \smallsetminus \Ind f})^{-1}(p)$, we choose distinct hyperplanes $H_1, H_2$ through $p$. We compute the intersection $f^* H_1 \cdot f^* H_2$, then remove points in $\Ind f$, which is the simultaneous vanishing locus of $f_0,f_1,f_2$. 

For each $p$, the above calculation gives a lower bound on $\deg_{\Top}(f)$. Taking $p = [1:1:1]$, we may check (e.g. by computer) that
$$\#(f|_{\PP^2 \smallsetminus \Ind f})^{-1}(p) = 5.$$
Now we show that $5$ is also an upper bound. Let $H_1, H_2$ be generic hyperplanes. By B\'ezout's theorem, since $f^* H_1$ and $f^* H_2$ are degree $5$ and share no component, they intersect with multiplicity $25$. Let $(f^* H_1 \cdot f^* H_2)_q$ denote the intersection multiplicity at $q \in \PP^2$. We claim that 
$$\sum_{q \in \Ind f} (f^* H_1 \cdot f^* H_2)_q \geq 20,$$ 
so there are at most $5$ points in $(f|_{\PP^2 \smallsetminus \Ind f})^{-1}(H_1 \cap H_2)$.
In fact, we claim that the following table describes the multiplicity of $(f^* H_1 \cdot f^* H_2)_q$ at each indeterminacy point $q$.
\begin{center}
\begin{tabular}{c|c}
    $q \in \Ind f$ & $(f^* H_1 \cdot f^* H_2)_q$ \\ \hline
    $[0:0:1]$ & $2$ \\
    $[-1:0:1]$ & $5$ \\
    $[1:i:0]$ & $2$ \\
    $[1:-i:0]$ & $2$ \\
    $[0:-1:1]$ & $2$ \\
    else & $1$
\end{tabular}
\end{center}
These calculations are routine but cumbersome. In light of the lower bound of $5$ on the generic topological degree,
$$\sum_{q \in \Ind f} (f^* H_1 \cdot f^* H_2)_q \leq 20,$$
so we need only check that the multiplicity at each of these points is at least the advertised value. The set $\Ind f$ is contained in the support of $f^* H_1 \cdot f^* H_2$, and so $f^* H_1 \cdot f^* H_2$ has multiplicity at least $1$ at each indeterminacy point. There are $12$ indeterminacy points. To confirm the higher multiplicities, we calculate the degree of the homogeneous ideal generated by $f^* H_1, f^*H_2$ and sufficiently high powers of the maximal ideal at each $q$.
\end{proof}

\begin{cor} \label{cor_unbounded_top}
    There is no uniform bound on the topological degree of a truly skew pentagram map $T^{(n)}_{a,b,c,d}$.
\end{cor}
\begin{proof}
    For all $n$, we have $\deg_{\Top} (T^{(8n)}_{0,2n,n,4n}) = (\deg_{\Top} (T^{(8)}_{0,2,1,4}))^n \geq 5^n$ by Lemma \ref{lem_basic} (\ref{it_lem_basic_top_deg}).
\end{proof}

\begin{question} 
Is there a convenient formula for $\deg_{\Top} (T^{(n)}_{a,b,c,d})$ in terms of $n, a, b, c, d$?
\end{question}

\begin{remark}
Geometrically, the contracted curves arise when the normalizing rotation $M$ in the proof of Proposition \ref{prop_formula} degenerates to a constant map.
\end{remark}

\begin{remark}
    The proof of Lemma \ref{lem_basic} (\ref{it_lem_basic_non_dominant}) showed that $f$ maps the line $L$ defined by $X + Y = Z$ onto itself. It is easy to confirm that $\deg_{\Top}(f|_L) = 2$, so $\deg_{\Top}(f) \geq 2$. Similarly, the line defined by $Z = 0$ is mapped onto itself $2$-to-$1$. Applying Lemma \ref{lem_subsystem} to either line then gives a quick proof that $\lambda_1(f) \geq 2$. 
\end{remark}

\begin{remark}
    In general, surface maps $g$ obey $\lambda_1(g)\geq \sqrt{\deg_{\Top}(g)}$ by log concavity of dynamical degrees, see e.g. \cite{MR4048444}. So Lemma \ref{lem_basic} shows that $\lambda_1(f) \geq \sqrt{5}$.
\end{remark}

\subsection{The dynamical degree}
As indicated in the above remarks, some exponential growth properties of this skew pentagram map  are already clear from basic geometric information about $f$. But our goal is to show that some skew pentagram maps $T_{a,b,c,d}$ satisfy $\lambda_1(T_{a,b,c,d}) = \deg T_{a,b,c,d} = 4$, which we interpret as ``maximal entropy'', rather than just positive entropy. In this subsection, we show that $\lambda_1(f) = 4$ by constructing an algebraically stable model. Then we use $\lambda_1(f)$ to show that $\lambda_1(T_{a,b,c,d}) = 4$ for any parameters $(a,b,c,d)$ that induce $\bar{T}_{\RS}$.

One may check directly that $p_1, p_2, p_3 \in \Ind f$, so all three contracted curves of $f$ witness the failure of algebraic stability. It turns out that we may resolve the failure of algebraic stability in the simplest way possible given this data -- by blowing up $\PP^2$ at $p_1, p_2, p_3$. For background on surface blowups, see e.g. \cite[Chapter V]{Hartshorne}.

Let $\pi: \mc{X} \to \PP^2$ be the blowup of $\PP^2$ at $p_1$, $p_2$, $p_3$. The exceptional divisors are denoted $E_1$, $E_2$, $E_3$, respectively. 

Let $\hat{f} : \mc{X} \dashrightarrow \mc{X}$ be the birational conjugate of $f$ to $\mc{X}$ by $\pi$, so we have a commutative diagram
\begin{equation}
    \begin{tikzcd}
\mathcal{X} \arrow[r, "\hat{f}", dashed] \arrow[d, "\pi"'] & \mathcal{X} \arrow[d, "\pi"] \\
\mathbb{P}^2 \arrow[r, "f"', dashed]                       & \mathbb{P}^2                
\end{tikzcd}
\label{eq_def_f_stable}
\end{equation}

The strict transform of any given curve $C \subset \PP^2$ to $\mc{X}$ is denoted $\hat{C}$.

\begin{lemma}
The rational map $\hat{f} : \mc{X} \dashrightarrow \mc{X}$ has no contracted curves, hence is algebraically stable.
\end{lemma}

\begin{proof}
    If $C$ is a contracted curve of $\hat{f}$, then either $C$ is an exceptional divisor of $\pi$, or $\pi(C)$ is a contracted curve of $f$. So the contracted curves of $\hat{f}$ are in the set $\mc{C} \colonequals \{E_1, E_2, E_3, \hat{C}_1, \hat{C}_2, \hat{C}_3\}$. No $C \in \mc{C}$ is contracted by $\hat{f}$. To see this, we observe that since each curve $C$ is rational, it suffices to parameterize each generically by a birational map $\gamma: \A^1 \dashrightarrow C$ and check that each composition $\hat{f} \circ \gamma$ has nonconstant image. Table \ref{table_contracted} summarizes the results of the check for each curve.

    For each $\hat{C}_i$, since $f(C_i \smallsetminus \Ind f) = p_i$, there is only one reasonable choice of local coordinate on the target: the affine patch containing $E_i$. For instance, to show $\hat{C}_1$ is not contracted, we let $\A^2$ be the main affine patch in $\PP^2$, and let $U = \A^2 \smallsetminus \{(0,0)\}$. Viewing $U$ as a subset of $\mc{X}$, the affine curve $\hat{C}_1|_U = C_1|_U$ is parametrized generically by $\gamma(t) = (t,-1-t)$. Then in coordinates $x$, $y/x$ on $\mc{X}$, the map $\hat{f} \circ \gamma : \A^1 \dashrightarrow \mc{X}$ is given by
    $$t \mapsto \left(0, \frac{t + 1}{2t^2 + t + 1} \right).$$
    So $\hat{f}$ generically maps $\hat{C}_1 \smallsetminus \Ind \hat{f}$ onto $E_1$, with degree $2$. 
    
    The local coordinates for the $E_i$ are less obvious, but trial and error with different local coordinate systems leads us to find that each is mapped generically onto itself $2$-to-$1$, as indicated in Table \ref{table_contracted}. The rows of $\hat{C}_3$ and $E_3$ follow from the rows of $\hat{C}_2$ and $E_2$ by complex conjugation. 
    \begin{table}[h]
        \centering 
        \def\arraystretch{2}
        \tiny{
        \begin{tabular}{c|c|c|c|c|c}
            $C$ & $U \subset \mc{X}$ & $\gamma(t)$  & Local coordinates  & $(\hat{f} \circ \gamma)(t)$ & $\ovl{f(C \smallsetminus \Ind f)} $\\
            &&&on target $\mc{X}$ & \\ \hline
            
            $\hat{C}_1$ & $x = X/Z, y = Y/Z $ & $(t,-1-t)$ & $x, y/x$ & $ \left(0,\frac{t + 1}{2t^2 + t + 1}\right)$ & $E_1$  \\
            & $(x,y) \neq (0,0)$ &&& & \\ \hline
            $\hat{C}_2$ & $x = X/Z, y = Y/Z $ & $((1+2i)t-1,t)$ & $Z/X, (Y - iX)/Z$ & $\left(0,\frac{{\left(6 i - 8\right) t^{3} + \left(-17 i - 9\right) t^{2} + \left(-3 i + 11\right) t + 2 i}}{\left(-2 i - 4\right) t^{2} + 2 t} \right)$ & $E_2$ \\
                        & $(x,y) \neq (0,0)$ &&&& \\\hline
            $\hat{C}_3$ & $x = X/Z, y = Y/Z $ & $((1-2i)t-1,t)$ & $Z/X, (Y + iX)/Z$ & $ \left(0,\frac{\left(-6 i - 8\right) t^{3} + \left(17 i - 9\right) t^{2} + \left(3 i + 11\right) t - 2 i}{\left(2 i - 4\right) t^{2} + 2 t} \right)$ & $E_3$ \\
                        & $(x,y) \neq (0,0)$ &&&& \\\hline
            $E_1$ & $x, y/x$ & $(0,t)$ & $x, y/x$ & $\left(0, \frac{-t - 1}{t^2 - t + 2}\right)$   & $E_1$
            \\ \hline
            $E_2$ & $Z/X, (Y - iX)/Z$ & $(0,t)$ & $Z/X, (Y-iX)/Z$ & $\left(0, \frac{\left(i - 1\right) t^{2} + \left(i - 2\right) t}{\left(-i + \right) t - i} \right)$ & $E_2$ 
            \\ \hline
            $E_3$ & $Z/X, (Y + iX)/Z$ & $(0,t)$ & $Z/X, (Y+iX)/Z$ & $\left(0, \frac{\left(- i - 1\right) t^{2} + \left(- i - 2\right) t}{\left(i + 1\right) t + i} \right)$ & $E_3$ 
            \\
        \end{tabular}
        }
        \caption{Verification that $\hat{f}: \mc{X} \dashrightarrow \mc{X}$ has no contracted curves.}
        \label{table_contracted}
    \end{table}
\end{proof}

Because $\hat{f}$ is algebraically stable, we can compute its dynamical degree by looking at the associated pullback homorphism $\hat{f}^*$. Our next lemma computes a matrix for $\hat{f}^*$.

Let $\Num \mc{X}$ denote the group of divisor classes on $\mc{X}$ up to numerical equivalence. As before, let $[H]$ be the hyperplane class in $\Num \PP^2$. Let $\hat{H}$ be the strict transform of $H$ in $\mc{X}$. We fix the choice of basis
$$\pi^* [H], \quad [E_1], \quad [E_2], \quad [E_3]$$
on $\Num \mc{X}$, identifying $\Num \mc{X} \cong \Z^4$.

\begin{lemma} \label{lem_ac_on_coh}
    The pullback homomorphism $\hat{f}^*: \Num \mc{X} \to \Num \mc{X}$ has matrix
    \begin{equation}
    \hat{f}^* = \begin{pmatrix}
        5 & 1 & 1 & 1 \\
        -1 & 1 & 0 & 0 \\
        -1 & 0 & 1 & 0 \\
        -1 & 0 & 0 & 1
    \end{pmatrix}.
    \label{eq_fhatstar}
    \end{equation}
    in the basis $\pi^*[H], [E_1], [E_2], [E_3]$.
\end{lemma}

\begin{proof} 
We apply Lemma \ref{lem_compute_pback} to compute the pullback associated to this surface map. Compare \cite[Lemma 4.1]{KR}.

\boxed{\textrm{Matrix Column 1:}}
 Let $H \in \Div \PP^2$ be a general hyperplane defined by $\alpha_0 X + \alpha_1 Y + \alpha_2 Z = 0$.  Then $f^* H$ is defined by vanishing of the degree $5$ polynomial $\alpha_0 f_0 + \alpha_1 f_1 + \alpha_2 f_2$. Since $(\pi^{-1} \circ f \circ \pi)^* \hat{H} = (\pi^{-1} \circ f)^* H$, and since $p_1, p_2, p_3 \in \Ind f$, we have that $(\pi^{-1} \circ f)^* H$ is the strict transform of $f^* H$. To express this divisor in our chosen basis, we compute the multiplicities of each exceptional divisor in $\pi^* f^* H$. We have 
 $$\mult_{p_1} f_0 = 2, \quad \mult_{p_1} f_1 = 2, \quad \mult_{p_1} f_2 = 1,$$
 $$\mult_{p_2} f_0 = 1, \quad \mult_{p_2} f_1 = 1, \quad \mult_{p_2} f_2 = 2,$$
 so by valuation properties, $\mult_{p_1} f^* H = 1$ and $\mult_{p_2} f^* H = 1$. 
 By $(\pm i)$-symmetry, $\mult_{p_3} f^* H = 1$. So $\hat{f}^* [\hat{H}] = 5[\hat{H}] - [E_1] - [E_2] - [E_3]$. 

\boxed{\textrm{Matrix Column 2:}}
 We have $f^{(-1)}(E_1) = C_1 \cup E_1$, by Lemma \ref{lem_basic} (\ref{it_lem_basic_contracted}) and the rightmost column of Table \ref{table_contracted}. So by Lemma \ref{lem_compute_pback}, we have $\hat{f}^* E_1 = m_{\hat{C}_1} \hat{C}_1 + m_{E_1} E_1$ for some local vanishing orders $m_{\hat{C}_1}, m_{E_1} \in \N$. In local coordinates $u' = x', v' = y'/x'$ on the codomain, a local equation for $E_1$ is $u' = 0$. The numerator in $u' = f_0(X,Y,Z)/f_2(X,Y,Z)$ in lowest terms contains the defining polynomial $X+Y+Z$ of $\hat{C}_1$ to multiplicity $1$, so $m_{\hat{C}_1} \leq 1$.
 In local coordinates $u = x, v = y/x$ on the domain, the numerator in $u' = f_0(u,uv,1)/f_2(u,uv,1)$ has factor $u$ with multiplicity $1$, so $m_{E_1} \leq 1$. So $\hat{f}^* E_1 = \hat{C}_1 + E_1$. 
 
 Since $C_1, C_2, C_3$ are all disjoint from $\{p_1,p_2,p_3\}$, we have $[\hat{C}_1] = [\hat{C}_2] = [\hat{C}_3] = \pi^* [H]$, whence $\hat{f}^* [E_1] = \pi^*[H] + [E_1]$.

\boxed{\textrm{Matrix Column 3:}} The same logic shows that $\hat{f}^* E_2 = m_{\hat{C}_2} \hat{C}_2 + m_{E_2} E_2$ for some local vanishing orders $m_{\hat{C}_2}, m_{E_2} \in \N$, and one may readily compute that each vanishing order is $1$.

 \boxed{\textrm{Matrix Column 4:}} Follows from Column 3 by complex conjugation.
\end{proof}

\begin{prop} \label{prop_rs_dd}
    The dynamical degrees of the skew pentagram map on rotationally symmetric octagons are
    $$\lambda_1(\bar{T}_{\RS}) = 4,$$
    $$\lambda_2(\bar{T}_{\RS}) = 5.$$
\end{prop}
\begin{proof}
    Dynamical degrees are birational invariants (Lemma \ref{lem_bir_inv}), so it suffices to compute $\lambda_1(\hat{f})$ and $\lambda_2(f)$.  Since $\hat{f}$ has no contracted curves, it is algebraically stable (Lemma \ref{lem_as_no_ctrcted}) and thus $\lambda_1(\hat{f}) = \rad \hat{f}^*$ by Lemma \ref{lem_norms_and_rads}. Examining \eqref{eq_fhatstar}, the eigenvalues of $\hat{f}^*$ are $4,2,1,1$, so $\rad \hat{f}^* = 4$. The second dynamical degree of a surface map is just its topological degree (Lemma \ref{lem_deg_top}), and $\deg_{\Top}(f) = 5$ by Lemma \ref{lem_basic}(\ref{it_lem_basic_top_deg}).

\end{proof}
 
If $g: X \dashrightarrow X$ is a dominant rational self map of a smooth projective surface of dimension $N$, the last dynamical degree $\lambda_N(g)$ is simply the (generic) topological degree of $g$. If $\lambda_N(g) = \max_{0 \leq i \leq N} \lambda_i(g)$, we say that $g$ has \emph{dominant topological degree}. These maps are chaotic, as made precise by Corollary \ref{cor_rs_consequences}. 

\begin{cor} \label{cor_rs_consequences}
\leavevmode
\begin{enumerate}
    \item The skew pentagram map $\bar{T}_{\RS}$ on rotationally symmetric $8$-gon classes has dominant topological degree.
    \item The map $\bar{T}_{\RS} : \mc{C}_{8,\RS} \dashrightarrow \mc{C}_{8,\RS}$ does not preserve any nonconstant fibration (Theorem \ref{thm_no_fib}).
    \item The topological entropy of $\bar{T}_{\RS}$ is $\log 5$.
    \item There is a unique measure of maximal entropy for $\bar{T}_{\RS}$, and the sequence of atomic measures associated to repelling points of period $n \to \infty$ equidistributes to it.
\end{enumerate}
\end{cor}

\begin{proof} \leavevmode
    \begin{enumerate}
        \item We trivially have $\lambda_0(\ovl{T}_\RS) = 1$ by Lemma \ref{lem_deg_top}, and the remaining dynamical degrees are computed in Proposition \ref{prop_rs_dd}.
        \item We have $\lambda_0(\bar{T}_{\RS}) \nmid \lambda_1(\bar{T}_{\RS})$ and $\lambda_1(\bar{T}_{\RS}) \nmid \lambda_2(\bar{T}_{\RS})$, so this follows from the product formula, Lemma \ref{lemma_rel_dd} (\ref{it_lemma_rel_dd_divide}). 
        \item The topological entropy of a dominant rational self-map $g$ of a smooth projective variety over $\C$ is at most the logarithm of its largest dynamical degree \cite[Theorem 1.4]{MR2119243}, and if the map has dominant topological degree, we have equality \cite{MR3436236}. (This result also appears in \cite{guedj}, but the proof has a gap; see \cite{MR3436236} for details.)
        \item This is also part of the main theorem of \cite{MR3436236}, building on \cite{guedj}.
    \end{enumerate}
\end{proof}

We now have all the ingredients required for an entropy sandwich, allowing us to prove our main non-integrability result, Theorem \ref{thm_main} (\ref{it_thm_main_nonint}). We state the more precise version as follows.
\begin{thm} \label{thm_main_body}
Let $a,b,c,d \in \Z$ be integers such that $a \equiv 0, b \equiv 2, c \equiv 1, d \equiv 4 \pmod 8.$ Then
$$\lambda_1(T_{a,b,c,d}) = 4.$$
\end{thm}

\begin{proof}
By Lemma \ref{lem_iterable}, the lattice map $T_{a,b,c,d}$ is iterable, so the dynamical degree exists. We compute it by finding equal upper and lower bounds, as follows:
\begin{align*}
    4 &=\lambda_1(\bar{T}_{\RS})&\text{(Proposition \ref{prop_rs_dd})}\\
    &\leq \lambda_1(T_{\RS})&\text{(Lemma \ref{lemma_rel_dd} (\ref{it_lemma_rel_dd_skew_prod}))}\\
    &\leq \lambda_1(T^{(8)}_{0,2,1,4})&\text{(Lemma \ref{lem_subsystem})}\\
    &\leq \lambda_1(T_{a,b,c,d})&\text{(Lemma \eqref{lem_dd_periodic_comparison})}\\
    &\leq 4&\text{(Proposition \ref{prop_upper})}.
\end{align*}
\end{proof}

Together with our analysis of equal-length maps in Theorem \ref{thm_equal_length}, this completes the proof of Theorem \ref{thm_main}. By considering specific spaces of polygons, we can then deduce Corollary \ref{cor_height}.

\begin{proof}[Proof of Corollary \ref{cor_height}]
      Under the same hypotheses on $a,b,c,d$ as in Theorem \ref{thm_main_body}, for any $n$, the restriction $T^{(8n)}_{a,b,c,d}$ to $8n$-gons is dominant by Proposition \ref{lem_dom_quadrilateral_trick}. Periodically embedding $(\PP^2)^8$ in $(\PP^2)^{8n}$ yields
     $$4 = \lambda_1(T^{(8)}_{0,2,1,4}) \leq \lambda_1(T^{(8n)}_{a,b,c,d}) = 4.$$
     By \cite[Theorem 1.2]{MR4958582}, we have
     $\lambda_1(T^{(8n)}_{a,b,c,d}) = \alpha_1(T^{(8n)}_{a,b,c,d})$.
\end{proof}

\subsection{A note on the heat map}

We crystallize our general approach in the following proposition.

\begin{prop} \label{prop_sandwich}
    Let $T: (\PP^N)^\Z \dashrightarrow (\PP^N)^\Z$ be an iterable lattice map on $(\PP^N)^\Z$. Suppose that there is some value of $n$ such that the map $T^{(n)}$ on $n$-periodic sequences is iterable and $\lambda_1( T^{(n)}) \geq \deg T$. Then 
    $$\lambda_1(T) = \lambda_1(T^{(n)}) = \deg T.$$ 
\end{prop}

\begin{proof}
Iterability ensures that both these dynamical degrees exist. We know $\lambda_1(T^{(n)}) \leq \lambda_1(T)$ by Lemma \ref{lem_dd_periodic_comparison}, and $\lambda_1(T) \leq \deg T$ by Lemma \ref{lem_first_iterate}. So the condition $\deg T \leq \lambda_1(T^{(n)}) $ creates a circle of inequalities.
\end{proof}

The \emph{heat map} $\mc{H}: (\PP^2)^\Z \dashrightarrow (\PP^2)^\Z$ is the lattice map given by the local rule
$$ \psi : (\PP^2)^4 \dashrightarrow \PP^2,$$
$$ v'_i = \ovl{(\ovl{v_i v_{i+2}} \cap \ovl{v_{i+1} v_{i+3}}) (\ovl{v_i v_{i+1}} \cap \ovl{v_{i + 2} v_{i + 3}})}.$$
The point $v'_i$ is interpreted as the ``projective midpoint'' of $v_{i+1}$ and $v_{i+2}$. Kaschner-Roeder and Schwartz studied the heat map on the moduli space of closed pentagons \cite{MR3642552, KR}.

\begin{cor} \label{cor_heat}
    The dynamical degree of the complex heat map $\mc{H}$ is $4$.
\end{cor}

\begin{proof}
    The heat map respects the diagonal action of $\PGL_3$, and fixes regular real $5$-gons up to dilatation, so it is iterable over $\C$. Kaschner-Roeder showed that $\ovl{\mc{H}}^{(5)}$ is dominant and that $\lambda_1(\ovl{\mc{H}}^{(5)}) = 4$ \cite{KR}. By semiconjugacy, we have $4 \leq \lambda_1(\mc{H}^{(5)})$.
    It is easy to check by computer that $\deg \psi = 4$; see \cite{code}. So $\deg \mc{H} \leq \lambda_1(\mc{H}^{(5)})$, meeting the hypotheses of Proposition \ref{prop_sandwich}.
\end{proof}

\begin{question}
It appears that the heat map shares many features with skew pentagram maps. Not only do they have the same degree and dynamical degree, and share a projective equivariance property, but the particular cases under study each have dominant topological degree and admit a contracted-curve-free model \cite{KR}. Further, in both cases, the algebraically stable model has a dominant eigenvector given by the anticanonical class of the model. Why?
\end{question}

\bibliographystyle{alpha}
\bibliography{bib}

@misc{textbookcasepentagramrigidity,
      title={A Textbook Case of Pentagram Rigidity}, 
      author={Richard Evan Schwartz},
      year={2021},
      eprint={2108.07604},
      archivePrefix={arXiv},
      primaryClass={math.DS},
      url={https://arxiv.org/abs/2108.07604},
     note={arXiv:2108.07604},
}

@article {MR4761767,
    AUTHOR = {Schwartz, Richard Evan},
     TITLE = {Pentagram rigidity for centrally symmetric octagons},
   JOURNAL = {Int. Math. Res. Not. IMRN},
  FJOURNAL = {International Mathematics Research Notices. IMRN},
      YEAR = {2024},
    NUMBER = {12},
     PAGES = {9535--9561},
      ISSN = {1073-7928,1687-0247},
   MRCLASS = {37D40 (37C85 51M04)},
  MRNUMBER = {4761767},
MRREVIEWER = {Boris\ A.\ Khesin},
       DOI = {10.1093/imrn/rnae050},
       URL = {https://doi.org/10.1093/imrn/rnae050},
}

@manual{sage,
  Key          = {SageMath},
  Author       = {The {Sage Developers}},
  Title        = {{S}ageMath, the {S}age {M}athematics {S}oftware {S}ystem ({V}ersion 10.0)},
  note         = {{\tt https://www.sagemath.org}},
  Year         = {2023},
}

@Book{Silverman10,
    author = {Joseph H. Silverman},
    title = {The Arithmetic of Dynamical Systems},
    series = {Graduate Texts in Mathematics},
    volume = {241},
    publisher = {Springer-Verlag},
    year = {2007},
}

@article{OST,
  title={The pentagram map: a discrete integrable system},
  author={Ovsienko, Valentin and Schwartz, Richard and Tabachnikov, Serge},
  journal={Communications in Mathematical Physics},
  volume={299},
  number={2},
  pages={409--446},
  year={2010},
  publisher={Springer}
}

@article {MR2119243,
    AUTHOR = {Dinh, Tien-Cuong and Sibony, Nessim},
     TITLE = {Regularization of currents and entropy},
   JOURNAL = {Ann. Sci. \'Ecole Norm. Sup. (4)},
  FJOURNAL = {Annales Scientifiques de l'\'Ecole Normale Sup\'erieure.
              Quatri\`eme S\'erie},
    VOLUME = {37},
      YEAR = {2004},
    NUMBER = {6},
     PAGES = {959--971},
      ISSN = {0012-9593},
   MRCLASS = {32U40 (32C30 32H04 32Q15 37B40)},
  MRNUMBER = {2119243},
MRREVIEWER = {Ma\l gorzata\ Stawiska},
       DOI = {10.1016/j.ansens.2004.09.002},
       URL = {https://doi.org/10.1016/j.ansens.2004.09.002},
}

@article {dinhnguyen,
    AUTHOR = {Dinh, Tien-Cuong and Nguy\^en, Vi\^et-Anh},
     TITLE = {Comparison of dynamical degrees for semi-conjugate meromorphic
              maps},
   JOURNAL = {Comment. Math. Helv.},
  FJOURNAL = {Commentarii Mathematici Helvetici. A Journal of the Swiss
              Mathematical Society},
    VOLUME = {86},
      YEAR = {2011},
    NUMBER = {4},
     PAGES = {817--840},
      ISSN = {0010-2571,1420-8946},
   MRCLASS = {32H50 (32U40 37F10)},
  MRNUMBER = {2851870},
MRREVIEWER = {Andreas\ H\"oring},
       DOI = {10.4171/CMH/241},
       URL = {https://doi.org/10.4171/CMH/241},
}

@article {OSTclosed,
    AUTHOR = {Ovsienko, Valentin and Schwartz, Richard Evan and Tabachnikov,
              Serge},
     TITLE = {Liouville-{A}rnold integrability of the pentagram map on
              closed polygons},
   JOURNAL = {Duke Math. J.},
  FJOURNAL = {Duke Mathematical Journal},
    VOLUME = {162},
      YEAR = {2013},
    NUMBER = {12},
     PAGES = {2149--2196},
      ISSN = {0012-7094,1547-7398},
   MRCLASS = {37J35 (51A20)},
  MRNUMBER = {3102478},
MRREVIEWER = {V.\ Oproiu},
       DOI = {10.1215/00127094-2348219},
       URL = {https://doi.org/10.1215/00127094-2348219},
}

@article{soloviev,
  title={Integrability of the pentagram map},
  author={Soloviev, Fedor},
  journal={Duke Mathematical Journal},
  volume={162},
  number={15},
  pages={2815--2853},
  year={2013},
  publisher={Duke University Press}
}

@article {izosimov,
    AUTHOR = {Izosimov, Anton},
     TITLE = {Pentagram maps and refactorization in {P}oisson-{L}ie groups},
   JOURNAL = {Adv. Math.},
  FJOURNAL = {Advances in Mathematics},
    VOLUME = {404},
      YEAR = {2022},
     PAGES = {Paper No. 108476, 46},
      ISSN = {0001-8708,1090-2082},
   MRCLASS = {37J70 (39A70)},
  MRNUMBER = {4430020},
MRREVIEWER = {Alberto\ Medina},
       DOI = {10.1016/j.aim.2022.108476},
       URL = {https://doi-org.ezp-prod1.hul.harvard.edu/10.1016/j.aim.2022.108476},
}

@book{Hartshorne,
  title={Algebraic Geometry},
  author={Hartshorne, Robin},
  isbn={9781475738490},
  series={Graduate Texts in Mathematics},
  url={https://books.google.com/books?id=7z4mBQAAQBAJ},
  year={2013},
  publisher={Springer New York}
}

@article {weinreich2021,
    AUTHOR = {Weinreich, Max},
     TITLE = {The algebraic dynamics of the pentagram map},
   JOURNAL = {Ergodic Theory Dynam. Systems},
  FJOURNAL = {Ergodic Theory and Dynamical Systems},
    VOLUME = {43},
      YEAR = {2023},
    NUMBER = {10},
     PAGES = {3460--3505},
      ISSN = {0143-3857,1469-4417},
   MRCLASS = {37J70 (14E05 14H70 37P05)},
  MRNUMBER = {4637161},
MRREVIEWER = {Alexandr\ V.\ Pukhlikov},
       DOI = {10.1017/etds.2022.82},
       URL = {https://doi.org/10.1017/etds.2022.82},
}

@article{schwartz,
author = "Schwartz, Richard",
fjournal = "Experimental Mathematics",
journal = "Experiment. Math.",
number = "1",
pages = "71--81",
publisher = "A K Peters, Ltd.",
title = "The pentagram map",
url = "https://projecteuclid.org:443/euclid.em/1048709118",
volume = "1",
year = "1992"
}

@article{khesinsoloviev,
  title={Integrability of higher pentagram maps},
  author={B. Khesin and F. Soloviev},
  journal={Mathematische Annalen},
  year={2012},
  volume={357},
  pages={1005-1047}
}

@book{GIT,
  title={Geometric Invariant Theory},
  author={Mumford, David and Fogarty, John and Kirwan, Frances},
  number={v. 34},
  isbn={9783540569633},
  lccn={93033772},
  series={Ergebnisse der Mathematik und Ihrer Grenzgebiete, 3 Folge/A Series of Modern Surveys in Mathematics Series},
  url={https://books.google.com/books?id=dFlv3zn\_2-gC},
  year={1994},
  publisher={Springer Berlin Heidelberg}
}

@article {MR3342248,
    AUTHOR = {Roeder, Roland K. W.},
     TITLE = {The action on cohomology by compositions of rational maps},
   JOURNAL = {Math. Res. Lett.},
  FJOURNAL = {Mathematical Research Letters},
    VOLUME = {22},
      YEAR = {2015},
    NUMBER = {2},
     PAGES = {605--632},
      ISSN = {1073-2780},
   MRCLASS = {32H50 (37F10 55N45)},
  MRNUMBER = {3342248},
MRREVIEWER = {Vincent Guedj},
       DOI = {10.4310/MRL.2015.v22.n2.a13},
       URL = {https://doi.org/10.4310/MRL.2015.v22.n2.a13},
}

@article {MR4048444,
    AUTHOR = {Truong, Tuyen Trung},
     TITLE = {Relative dynamical degrees of correspondences over a field of
              arbitrary characteristic},
   JOURNAL = {J. Reine Angew. Math.},
  FJOURNAL = {Journal f\"{u}r die Reine und Angewandte Mathematik. [Crelle's
              Journal]},
    VOLUME = {758},
      YEAR = {2020},
     PAGES = {139--182},
      ISSN = {0075-4102},
   MRCLASS = {37P05 (14E05 37F05 37F80)},
  MRNUMBER = {4048444},
MRREVIEWER = {Yu Yasufuku},
       DOI = {10.1515/crelle-2017-0052},
       URL = {https://doi.org/10.1515/crelle-2017-0052},
}

@article {MR3436236,
    AUTHOR = {Dinh, Tien-Cuong and Nguy\^en, Vi\^et-Anh and Truong, Tuyen
              Trung},
     TITLE = {Equidistribution for meromorphic maps with dominant
              topological degree},
   JOURNAL = {Indiana Univ. Math. J.},
  FJOURNAL = {Indiana University Mathematics Journal},
    VOLUME = {64},
      YEAR = {2015},
    NUMBER = {6},
     PAGES = {1805--1828},
      ISSN = {0022-2518,1943-5258},
   MRCLASS = {32H50 (37F10)},
  MRNUMBER = {3436236},
MRREVIEWER = {Fei\ Hu},
       DOI = {10.1512/iumj.2015.64.5674},
       URL = {https://doi.org/10.1512/iumj.2015.64.5674},
}

@article {guedj,
    AUTHOR = {Guedj, Vincent},
     TITLE = {Ergodic properties of rational mappings with large topological
              degree},
   JOURNAL = {Ann. of Math. (2)},
  FJOURNAL = {Annals of Mathematics. Second Series},
    VOLUME = {161},
      YEAR = {2005},
    NUMBER = {3},
     PAGES = {1589--1607},
      ISSN = {0003-486X,1939-8980},
   MRCLASS = {32H50 (32H30 37A20 37F10)},
  MRNUMBER = {2179389},
MRREVIEWER = {Jean-Yves\ Briend},
       DOI = {10.4007/annals.2005.161.1589},
       URL = {https://doi.org/10.4007/annals.2005.161.1589},
}

@book {MR3642552,
    AUTHOR = {Schwartz, Richard Evan},
     TITLE = {The projective heat map},
    SERIES = {Mathematical Surveys and Monographs},
    VOLUME = {219},
 PUBLISHER = {American Mathematical Society, Providence, RI},
      YEAR = {2017},
     PAGES = {x+195},
      ISBN = {978-1-4704-3514-1},
   MRCLASS = {37-02 (14E05 26A18 37B05 37E30 51M15)},
  MRNUMBER = {3642552},
MRREVIEWER = {Roland\ K. W. Roeder},
       DOI = {10.1090/surv/219},
       URL = {https://doi.org/10.1090/surv/219},
}

@article {KR,
    AUTHOR = {Kaschner, Scott R. and Roeder, Roland K. W.},
     TITLE = {Complex perspective for the projective heat map acting on
              pentagons},
   JOURNAL = {Conform. Geom. Dyn.},
  FJOURNAL = {Conformal Geometry and Dynamics. An Electronic Journal of the
              American Mathematical Society},
    VOLUME = {21},
      YEAR = {2017},
     PAGES = {247--263},
      ISSN = {1088-4173},
   MRCLASS = {32H50 (26A18 37F10 37F15 51M15)},
  MRNUMBER = {3645773},
MRREVIEWER = {Mattias\ Jonsson},
       DOI = {10.1090/ecgd/310},
       URL = {https://doi.org/10.1090/ecgd/310},
}

@article {MR4958582,
    AUTHOR = {Matsuzawa, Yohsuke and Xie, Junyi},
     TITLE = {Arithmetic degree and its application to {Z}ariski dense orbit
              conjecture},
   JOURNAL = {J. Lond. Math. Soc. (2)},
  FJOURNAL = {Journal of the London Mathematical Society. Second Series},
    VOLUME = {112},
      YEAR = {2025},
    NUMBER = {3},
     PAGES = {Paper No. e70282},
      ISSN = {0024-6107,1469-7750},
   MRCLASS = {37P15 (37P55)},
  MRNUMBER = {4958582},
       DOI = {10.1112/jlms.70282},
       URL = {https://doi.org/10.1112/jlms.70282},
}

@article {MR4831038,
    AUTHOR = {Jia, Jia and Shibata, Takahiro and Xie, Junyi and Zhang,
              De-Qi},
     TITLE = {Endomorphisms of quasi-projective varieties: towards {Z}ariski
              dense orbit and {K}awaguchi-{S}ilverman conjectures},
   JOURNAL = {Math. Res. Lett.},
  FJOURNAL = {Mathematical Research Letters},
    VOLUME = {31},
      YEAR = {2024},
    NUMBER = {3},
     PAGES = {701--746},
      ISSN = {1073-2780,1945-001X},
   MRCLASS = {37P05 (14E05 37F80)},
  MRNUMBER = {4831038},
       DOI = {10.4310/mrl.241113041354},
       URL = {https://doi.org/10.4310/mrl.241113041354},
}

@misc{zou2025spiralstictactoepartitiondeep,
      title={Spirals, Tic-Tac-Toe Partition, and Deep Diagonal Maps}, 
      author={Zhengyu Zou},
      year={2025},
      eprint={2412.15561},
      archivePrefix={arXiv},
      primaryClass={math.DS},
      url={https://arxiv.org/abs/2412.15561}, 
      note={arXiv:2412.15561},
}

@article {dangsingular,
    AUTHOR = {Dang, Nguyen-Bac},
     TITLE = {Degrees of iterates of rational maps on normal projective
              varieties},
   JOURNAL = {Proc. Lond. Math. Soc. (3)},
  FJOURNAL = {Proceedings of the London Mathematical Society. Third Series},
    VOLUME = {121},
      YEAR = {2020},
    NUMBER = {5},
     PAGES = {1268--1310},
      ISSN = {0024-6115,1460-244X},
   MRCLASS = {14E05 (37C35 37P55)},
  MRNUMBER = {4133708},
MRREVIEWER = {Alexandr\ V.\ Pukhlikov},
       DOI = {10.1112/plms.12366},
       URL = {https://doi.org/10.1112/plms.12366},
}

@article {MR3438382,
    AUTHOR = {Khesin, Boris and Soloviev, Fedor},
     TITLE = {The geometry of dented pentagram maps},
   JOURNAL = {J. Eur. Math. Soc. (JEMS)},
  FJOURNAL = {Journal of the European Mathematical Society (JEMS)},
    VOLUME = {18},
      YEAR = {2016},
    NUMBER = {1},
     PAGES = {147--179},
      ISSN = {1435-9855},
   MRCLASS = {37K10 (37K25 53A20)},
  MRNUMBER = {3438382},
MRREVIEWER = {Yuri B. Suris},
       DOI = {10.4171/JEMS/586},
       URL = {https://doi-org.ezp-prod1.hul.harvard.edu/10.4171/JEMS/586},
}

@article {MR3282373,
    AUTHOR = {Khesin, Boris and Soloviev, Fedor},
     TITLE = {Non-integrability vs. integrability in pentagram maps},
   JOURNAL = {J. Geom. Phys.},
  FJOURNAL = {Journal of Geometry and Physics},
    VOLUME = {87},
      YEAR = {2015},
     PAGES = {275--285},
      ISSN = {0393-0440},
   MRCLASS = {37J35 (52A55)},
  MRNUMBER = {3282373},
       DOI = {10.1016/j.geomphys.2014.07.027},
       URL = {https://doi-org.ezp-prod1.hul.harvard.edu/10.1016/j.geomphys.2014.07.027},
}

@article {izo_khes,
    AUTHOR = {Izosimov, Anton and Khesin, Boris},
     TITLE = {Long-diagonal pentagram maps},
   JOURNAL = {Bull. Lond. Math. Soc.},
  FJOURNAL = {Bulletin of the London Mathematical Society},
    VOLUME = {55},
      YEAR = {2023},
    NUMBER = {3},
     PAGES = {1314--1329},
      ISSN = {0024-6093,1469-2120},
   MRCLASS = {37J35 (37K25 53A20)},
  MRNUMBER = {4605672},
MRREVIEWER = {W.-H.\ Steeb},
       DOI = {10.1112/blms.12792},
       URL = {https://doi.org/10.1112/blms.12792},
}

@article {MR3356734,
    AUTHOR = {Mar\'{\i} Beffa, Gloria},
     TITLE = {On integrable generalizations of the pentagram map},
   JOURNAL = {Int. Math. Res. Not. IMRN},
  FJOURNAL = {International Mathematics Research Notices. IMRN},
      YEAR = {2015},
    NUMBER = {12},
     PAGES = {3669--3693},
      ISSN = {1073-7928},
   MRCLASS = {37J35 (37D40 37K10 37K35)},
  MRNUMBER = {3356734},
MRREVIEWER = {Yuri B. Suris},
       DOI = {10.1093/imrn/rnu044},
       URL = {https://doi-org.ezp-prod1.hul.harvard.edu/10.1093/imrn/rnu044},
}

@article {MR2358970,
    AUTHOR = {Hasselblatt, Boris and Propp, James},
     TITLE = {Degree-growth of monomial maps},
   JOURNAL = {Ergodic Theory Dynam. Systems},
  FJOURNAL = {Ergodic Theory and Dynamical Systems},
    VOLUME = {27},
      YEAR = {2007},
    NUMBER = {5},
     PAGES = {1375--1397},
      ISSN = {0143-3857},
   MRCLASS = {37F45 (32H50 37B40)},
  MRNUMBER = {2358970},
MRREVIEWER = {Mattias Jonsson},
       DOI = {10.1017/S0143385707000168},
       URL = {https://doi-org.ezp-prod1.hul.harvard.edu/10.1017/S0143385707000168},
}

@article {MR2317336,
    AUTHOR = {Speyer, David E.},
     TITLE = {Perfect matchings and the octahedron recurrence},
   JOURNAL = {J. Algebraic Combin.},
  FJOURNAL = {Journal of Algebraic Combinatorics. An International Journal},
    VOLUME = {25},
      YEAR = {2007},
    NUMBER = {3},
     PAGES = {309--348},
      ISSN = {0925-9899,1572-9192},
   MRCLASS = {05E99 (05C70)},
  MRNUMBER = {2317336},
MRREVIEWER = {Markus\ E.\ Fulmek},
       DOI = {10.1007/s10801-006-0039-y},
       URL = {https://doi.org/10.1007/s10801-006-0039-y},
}

@article {MR4033822,
    AUTHOR = {Galashin, Pavel and Pylyavskyy, Pavlo},
     TITLE = {Quivers with additive labelings: classification and algebraic
              entropy},
   JOURNAL = {Doc. Math.},
  FJOURNAL = {Documenta Mathematica},
    VOLUME = {24},
      YEAR = {2019},
     PAGES = {2057--2135},
      ISSN = {1431-0635,1431-0643},
   MRCLASS = {13F60 (05E40 37K10)},
  MRNUMBER = {4033822},
MRREVIEWER = {Changjian\ Fu},
}

@article {MR4790973,
    AUTHOR = {Arnold, Maxim and Schwartz, Richard Evan and Tabachnikov,
              Serge},
     TITLE = {On projective evolutes of polygons},
   JOURNAL = {Exp. Math.},
  FJOURNAL = {Experimental Mathematics},
    VOLUME = {33},
      YEAR = {2024},
    NUMBER = {3},
     PAGES = {347--356},
      ISSN = {1058-6458,1944-950X},
   MRCLASS = {51N35 (13F60 14H10 14H52 37D40)},
  MRNUMBER = {4790973},
       DOI = {10.1080/10586458.2022.2102095},
       URL = {https://doi.org/10.1080/10586458.2022.2102095},
}

@article {MR4646096,
    AUTHOR = {Affolter, Niklas Christoph and de Tili\`ere, B\'eatrice and
              Melotti, Paul},
     TITLE = {The {S}chwarzian octahedron recurrence (d{SKP} equation) {I}:
              explicit solutions},
   JOURNAL = {Comb. Theory},
  FJOURNAL = {Combinatorial Theory},
    VOLUME = {3},
      YEAR = {2023},
    NUMBER = {2},
     PAGES = {Paper No. 15, 58},
      ISSN = {2766-1334},
   MRCLASS = {82B23 (05A15 37K10 37K60 82B20)},
  MRNUMBER = {4646096},
       DOI = {10.5070/c63261993},
       URL = {https://doi.org/10.5070/c63261993},
}

@book {MR3617981,
    AUTHOR = {Eisenbud, David and Harris, Joe},
     TITLE = {3264 and all that---a second course in algebraic geometry},
 PUBLISHER = {Cambridge University Press, Cambridge},
      YEAR = {2016},
     PAGES = {xiv+616},
      ISBN = {978-1-107-60272-4; 978-1-107-01708-5},
   MRCLASS = {14-01 (14C15 14M15 14N10)},
  MRNUMBER = {3617981},
MRREVIEWER = {Arnaud\ Beauville},
       DOI = {10.1017/CBO9781139062046},
       URL = {https://doi-org.ezp-prod1.hul.harvard.edu/10.1017/CBO9781139062046},
}

@article {MR4865926,
    AUTHOR = {Affolter, Niklas Christoph and de Tili\`ere, B\'eatrice and
              Melotti, Paul},
     TITLE = {The {S}chwarzian octahedron recurrence (d{SKP} equation) {II}:
              {G}eometric systems},
   JOURNAL = {Discrete Comput. Geom.},
  FJOURNAL = {Discrete \& Computational Geometry. An International Journal
              of Mathematics and Computer Science},
    VOLUME = {73},
      YEAR = {2025},
    NUMBER = {2},
     PAGES = {370--436},
      ISSN = {0179-5376,1432-0444},
   MRCLASS = {37K10 (39A36 82B20)},
  MRNUMBER = {4865926},
       DOI = {10.1007/s00454-024-00640-2},
       URL = {https://doi.org/10.1007/s00454-024-00640-2},
}

@article {MR4827588,
    AUTHOR = {Gubbiotti, Giorgio and Kels, Andrew P. and Viallet, Claude-M.},
     TITLE = {Algebraic entropy for hex systems},
   JOURNAL = {Nonlinearity},
  FJOURNAL = {Nonlinearity},
    VOLUME = {37},
      YEAR = {2024},
    NUMBER = {12},
     PAGES = {Paper No. 125007, 24},
      ISSN = {0951-7715,1361-6544},
   MRCLASS = {37K60 (37B40 37F80 39A36)},
  MRNUMBER = {4827588},
       DOI = {10.1088/1361-6544/ad88cd},
       URL = {https://doi.org/10.1088/1361-6544/ad88cd},
}

@article {MR2481234,
    AUTHOR = {Viallet, Claude M.},
     TITLE = {Integrable lattice maps: {$Q_V$}, a rational version of
              {$Q_4$}},
   JOURNAL = {Glasg. Math. J.},
  FJOURNAL = {Glasgow Mathematical Journal},
    VOLUME = {51},
      YEAR = {2009},
    NUMBER = {A},
     PAGES = {157--163},
      ISSN = {0017-0895,1469-509X},
   MRCLASS = {37K10 (37K60 39A10 52C99)},
  MRNUMBER = {2481234},
       DOI = {10.1017/S0017089508004874},
       URL = {https://doi.org/10.1017/S0017089508004874},
}

@article {MR2392894,
    AUTHOR = {Hietarinta, Jarmo and Viallet, Claude},
     TITLE = {Searching for integrable lattice maps using factorization},
   JOURNAL = {J. Phys. A},
  FJOURNAL = {Journal of Physics. A. Mathematical and Theoretical},
    VOLUME = {40},
      YEAR = {2007},
    NUMBER = {42},
     PAGES = {12629--12643},
      ISSN = {1751-8113,1751-8121},
   MRCLASS = {37K10 (35Q53 37K60 39A10)},
  MRNUMBER = {2392894},
MRREVIEWER = {Matteo\ Petrera},
       DOI = {10.1088/1751-8113/40/42/S09},
       URL = {https://doi.org/10.1088/1751-8113/40/42/S09},
}

@article {MR4906745,
    AUTHOR = {Hietarinta, Jarmo},
     TITLE = {Degree growth of lattice equations defined on a {$3\times3$}
              stencil},
   JOURNAL = {Open Commun. Nonlinear Math. Phys.},
  FJOURNAL = {Open Communications in Nonlinear Mathematical Physics},
      YEAR = {2024},
     PAGES = {19},
      ISSN = {2802-9356},
   MRCLASS = {37K60 (35Q53)},
  MRNUMBER = {4906745},
}

@article {MR4007163,
    AUTHOR = {Benedetto, Robert and Ingram, Patrick and Jones, Rafe and
              Manes, Michelle and Silverman, Joseph H. and Tucker, Thomas
              J.},
     TITLE = {Current trends and open problems in arithmetic dynamics},
   JOURNAL = {Bull. Amer. Math. Soc. (N.S.)},
  FJOURNAL = {American Mathematical Society. Bulletin. New Series},
    VOLUME = {56},
      YEAR = {2019},
    NUMBER = {4},
     PAGES = {611--685},
      ISSN = {0273-0979,1088-9485},
   MRCLASS = {37P05 (11G50 37P15 37P20 37P25 37P30 37P45 37P55)},
  MRNUMBER = {4007163},
       DOI = {10.1090/bull/1665},
       URL = {https://doi.org/10.1090/bull/1665},
}

@article {MR1704282,
    AUTHOR = {Bellon, M. P. and Viallet, C.-M.},
     TITLE = {Algebraic entropy},
   JOURNAL = {Comm. Math. Phys.},
  FJOURNAL = {Communications in Mathematical Physics},
    VOLUME = {204},
      YEAR = {1999},
    NUMBER = {2},
     PAGES = {425--437},
      ISSN = {0010-3616,1432-0916},
   MRCLASS = {37D45 (37A35 37C99 39A12)},
  MRNUMBER = {1704282},
MRREVIEWER = {Remo\ Badii},
       DOI = {10.1007/s002200050652},
       URL = {https://doi.org/10.1007/s002200050652},
}

@article {MR3534837,
    AUTHOR = {Gekhtman, Michael and Shapiro, Michael and Tabachnikov, Serge
              and Vainshtein, Alek},
     TITLE = {Integrable cluster dynamics of directed networks and pentagram
              maps},
   JOURNAL = {Adv. Math.},
  FJOURNAL = {Advances in Mathematics},
    VOLUME = {300},
      YEAR = {2016},
     PAGES = {390--450},
      ISSN = {0001-8708,1090-2082},
   MRCLASS = {37J35 (13F60 37K10)},
  MRNUMBER = {3534837},
MRREVIEWER = {Yuri\ B.\ Suris},
       DOI = {10.1016/j.aim.2016.03.023},
       URL = {https://doi.org/10.1016/j.aim.2016.03.023},
}

@article {MR3483131,
    AUTHOR = {Glick, Max and Pylyavskyy, Pavlo},
     TITLE = {{$Y$}-meshes and generalized pentagram maps},
   JOURNAL = {Proc. Lond. Math. Soc. (3)},
  FJOURNAL = {Proceedings of the London Mathematical Society. Third Series},
    VOLUME = {112},
      YEAR = {2016},
    NUMBER = {4},
     PAGES = {753--797},
      ISSN = {0024-6115,1460-244X},
   MRCLASS = {13F60 (37J35 51A05)},
  MRNUMBER = {3483131},
MRREVIEWER = {Boris\ A.\ Khesin},
       DOI = {10.1112/plms/pdw007},
       URL = {https://doi.org/10.1112/plms/pdw007},
}

@article {MR4363581,
    AUTHOR = {Dang, Nguyen-Bac and Herrig, Thorsten},
     TITLE = {Dynamical degrees of automorphisms on abelian varieties},
   JOURNAL = {Adv. Math.},
  FJOURNAL = {Advances in Mathematics},
    VOLUME = {395},
      YEAR = {2022},
     PAGES = {Paper No. 108082, 43},
      ISSN = {0001-8708,1090-2082},
   MRCLASS = {11G10 (11R52 14K05)},
  MRNUMBER = {4363581},
MRREVIEWER = {Thomas\ Ward},
       DOI = {10.1016/j.aim.2021.108082},
       URL = {https://doi-org.ezp-prod1.hul.harvard.edu/10.1016/j.aim.2021.108082},
}

@code{code,
    AUTHOR = {Weinreich,Max},
    TITLE = {Degree growth of skew pentagram maps},
    note = {Git{H}ub Repository, \url{https://github.com/MaxWeinreich/SkewPentagramMaps}},
    YEAR = 2025,
}
\end{document}